\documentclass{amsart}

\usepackage[english]{babel}
\usepackage[latin1]{inputenc}
\usepackage[colorlinks]{hyperref}

\usepackage{pgfplots}
\pgfplotsset{compat=1.10}
\usepgfplotslibrary{fillbetween}
\usetikzlibrary{patterns}

\usepackage{todonotes}

\allowdisplaybreaks

\usepackage{amssymb,amsmath,amsthm,amsfonts, mathrsfs, mathtools}
\usepackage{graphicx}
\usepackage{enumitem}

\def\eps{\varepsilon}
\def\N{\mathbb{N}}
\def\R{\mathbb{R}}

\def\loc{\mathrm{loc}}
\usepackage{xcolor}
\usepackage[colorlinks]{hyperref}
\hypersetup{linkcolor=black,citecolor=black,filecolor=black,urlcolor=blue}

\usepackage{fullpage}
\usepackage{tikzsymbols}

\let\div\relax
\DeclareMathOperator{\div}{div}
\DeclareMathOperator{\dist}{dist}

\newcommand{\pa}{\partial}

\newtheorem{proposition}{Proposition}[section]
\newtheorem{theorem}[proposition]{Theorem}
\newtheorem{corollary}[proposition]{Corollary}
\newtheorem{lemma}[proposition]{Lemma}

\theoremstyle{definition}

\newtheorem{remark}[proposition]{Remark}
\numberwithin{equation}{section}

\newtheorem*{theorem_a}{Theorem A}
\newcommand{\mf}[1]{ \mathbf{#1}}

\newcommand{\supp}{\mathrm{supp}}

\makeatletter
\@namedef{subjclassname@2020}{%
  \textup{2020} Mathematics Subject Classification}
\makeatother

\title{Free boundary problems with long-range interactions:\\ uniform Lipschitz estimates in the radius}

\author[N. Soave]{Nicola Soave} \thanks{N.~Soave is partially supported by the INdAM-GNAMPA group (Italy).}
\address{Nicola Soave \newline \indent
	Dipartimento di Matematica,  Politecnico di Milano,  \newline \indent
	Via Edoardo Bonardi 9, 20133 Milano, Italy}
\email{nicola.soave@polimi.it}

\author[H. Tavares]{Hugo Tavares}  \thanks{H.~Tavares is partially supported by the Portuguese government through FCT-Funda\c c\~ao para a Ci\^encia e a Tecnologia, I.P., under the projects UID/MAT/04459/2020, PTDC/MAT-PUR/28686/2017 and PTDC/MAT-PUR/1788/2020.}
\address{Hugo Tavares \newline \indent CAMGSD and Mathematics Department,
\newline \indent Instituto Superior T\'ecnico, Universidade de Lisboa   \newline \indent
Av. Rovisco Pais, 1049-001 Lisboa, Portugal}
\email{hugo.n.tavares@tecnico.ulisboa.pt}

\author[A. Zilio]{Alessandro Zilio}\thanks{A.~Zilio is partially supported by the project ANR-18-CE40-0013 SHAPO financed by the French Agence Nationale de la Recherche (ANR)}
\address{Alessandro Zilio \newline \indent Universit\'e de Paris and Sorbonne Universit\'e, CNRS, \newline \indent Laboratoire Jacques-Louis Lions (LJLL), F-75006 Paris, France}
\email{azilio@math.univ-paris-diderot.fr, alessandro.zilio@u-paris.fr}
\begin{document}

\begin{abstract}
Consider the class of optimal partition problems with long range interactions
\[
\inf \left\{ \sum_{i=1}^k \lambda_1(\omega_i):\  (\omega_1,\ldots, \omega_k) \in \mathcal{P}_r(\Omega) \right\},
\]
where $\lambda_1(\cdot)$ denotes the first Dirichlet eigenvalue, and $\mathcal{P}_r(\Omega)$ is the set of open $k$-partitions of $\Omega$ whose elements are at distance at least $r$: $\dist(\omega_i,\omega_j)\geq r$ for every $i\neq j$. In this paper we prove optimal uniform bounds (as $r\to 0^+$) in $\mathrm{Lip}$--norm for the associated $L^2$--normalized eigenfunctions, connecting in particular the nonlocal case $r>0$ with the local one $r \to 0^+$.

The proof uses new pointwise estimates for eigenfunctions, a one-phase Alt-Caffarelli-Friedman and the Caffarelli-Jerison-Kenig monotonicity formulas, combined with elliptic and energy estimates. Our result extends to other contexts, such as singularly perturbed harmonic maps with distance constraints.
\end{abstract}

\date{\today}

\maketitle
\section{Introduction}

The purpose of this paper is to investigate uniform regularity estimates for a family of long-range (nonlocal) interaction problems. Let $\Omega$ be a smooth bounded domain of $\R^N$, $N \ge 2$ and $k \ge 2$ be integers. Given $r \ge 0$, we consider the set of all $k$-partitions of $\Omega$ whose elements are at distance at least $r$:
\[
\mathcal{P}_r(\Omega)=\left\{  (\omega_1,\ldots, \omega_k)\left| \begin{array}{l} \omega_i \subset \Omega \text{ is a nonempty open set for all $i$} ,\\ \omega_i \cap \omega_j = \emptyset \quad \text{and} \quad  \dist(\omega_i,\omega_j)\geq r\ \forall i\neq j\end{array}\right.\right\}
\]
(notice that the request that $\omega_i \cap \omega_j = \emptyset$ is redundant for $r>0$, but not for $r=0$). It is plain that there exists $\bar r>0$ such that $\mathcal{P}_r(\Omega) \neq \emptyset$, for every $r\in [0,\bar r)$. For any such $r$, we are concerned with the following optimization problem:
\begin{equation}\label{eqn eig dist}
c_r:=\inf \left\{ \sum_{i=1}^k \lambda_1(\omega_i):\  (\omega_1,\ldots, \omega_k) \in \mathcal{P}_r(\Omega) \right\},
\end{equation}
where $\lambda_1(\cdot)$ denotes the first Dirichlet eigenvalue. 

The short-range (local) case, corresponding to the choice $r=0$, is a typical example of optimal partition problem, a very active topic of research since the seminal paper \cite{BucurButtazzoHenrot}. Existence and properties of minimizers for $c_0$ are essentially understood: we collect in the following theorem what has been proved in \cite{CafLin06, CTVfuc, TT} (see also \cite{HHHT} and \cite{RTT}). 

\begin{theorem_a}\label{thm:local}
The optimal value $c_0$ is attained by a minimal partition $(\Omega_{1,0},\dots,\Omega_{k,0})$ which exhausts $\Omega$, in the sense that $\bigcup_i \overline{\Omega_{i,0}} = \overline{\Omega}$; moreover, the free boundary $\bigcup_i \partial \Omega_{i,0}$ consists of piece-wise $C^{1,\alpha}$-hypersurfaces of dimension $N-1$, up to a singular set of dimension $N-2$ (the singular set is actually discrete in dimension $N=2$). Finally, the  eigenfunctions $u_{i,0}$ associated with $\Omega_{i,0}$ are globally Lipschitz continuous, which is the optimal regularity in this case.
\end{theorem_a}
Finer results for the singular set are proved in the recent paper \cite{Alp}.

Much less is known in the nonlocal case $r>0$. In a joint paper with S. Terracini \cite{STTZ2018} (see Theorem 1.2 and Theorem 1.3-(3), (6) therein), we have shown the following properties.
\begin{enumerate}
\item \textit{Existence.} The level $c_r$ is achieved by an open optimal partition $(\Omega_{1,r},\ldots, \Omega_{k,r})$;
\item \textit{Exterior sphere condition and exact distance between the optimal sets.} Given $x_0\in \partial \Omega_{i,r} \setminus \partial \Omega$, there exists $j\neq i$ and $y_0\in \partial \Omega_{j,r}$ such that $|x_0-y_0|=r$, and $\Omega_{i,r} \cap B_r(y_0) = \emptyset$; in particular, $\dist(\Omega_{i,r},\Omega_{j,r})=r$ and each set $\Omega_{i,r}$ satisfies an exterior sphere condition of radius $r$ at any of its boundary point.
\end{enumerate}
The second statement together with \cite[Lemma 6.4]{CPQ} yields:
\begin{enumerate}
\item[(3)] \textit{Measure of the Free Boundary.} The sets $\partial \Omega_{i,r}$ have locally finite perimeter in $\Omega$. 
\end{enumerate}

The approach used both in the local \cite{CafLin06, CTVfuc, TT} and in the nonlocal case \cite{STTZ2018} consists in studying the following relaxed formulation of $c_r$ in terms of measurable functions rather than sets:
\begin{equation}\label{eq:weak_characterization}
\tilde c_r=\inf\left\{ \sum_{i=1}^k \int_\Omega |\nabla u_i|^2\left| \begin{array}{l} u_i \in H^1_0(\Omega), \ \int_\Omega u_i^2=1\ \forall i, \\ \int_{\Omega} u_i^2 u_j^2 = 0 \text{ and } \dist(\text{supp}\, u_i,\text{supp}\, u_j)\geq r,\ \forall i\neq j \end{array}\right. \right\}.
\end{equation}
It is shown that there exists a minimizer $\mf{u}_r = (u_{1,r}, \dots, u_{k,r})$ for $\tilde c_r$. Moreover:
\begin{enumerate}
\item[(a)] \textit{Optimal regularity.} Each $u_{i,r}$ is Lipschitz continuous in $\overline \Omega$. In particular, the positivity sets $\Omega_{i,r}:=\{u_{i,r}>0\}$ are open and $(\Omega_{1,r}, \dots, \Omega_{k,r})\in \mathcal{P}_r(\Omega)$;
\item[(b)] \textit{Equation of $u_{i,r}$.} $-\Delta u_{i,r}=\lambda_1(\Omega_{i,r})u_{i,r}$ in $\Omega_{i,r}$. The partition $(\Omega_{1,r},\ldots, \Omega_{2,r})$ achieves $c_r$, which coincides with $\tilde c_r$, and satisfies conditions (1)--(3).
\end{enumerate}
Under an additional regularity assumption of the free boundary $\partial \Omega_i$, we have also derived a free boundary condition, satisfied by the eigenfunctions of the optimal partitions (see \cite[Theorem 1.6]{STTZ2018}). The validity of such a condition remains a crucial open problem in the general setting for optimal partition problems with a distance constraint.

The techniques adopted in the local and nonlocal cases are completely different. Powerful tools typically employed in the former ones, such as monotonicity formulas, free boundary conditions and blow-up methods, cannot be adapted in the context of optimal partitions at distance, due to the nonlocal nature of the interaction between different densities/sets. This is why the free boundary regularity for problem $c_0$ is settled, while the same problem for $c_r$ is open. However, the common optimal Lipschitz regularity of $\mf{u}_r$ suggests that it should be possible to look at both problems, the local and the nonlocal ones, as a $1$-parameter family, where the parameter is the distance $r$ between the different supports. The main results of this paper establish that this is possible, at least at the level of the eigenfunctions. More precisely:

\begin{theorem}\label{thm:unif Lip}
There exists a constant $C > 0$ such that 
 \[
	\|\mathbf{u}_r\|_{\mathrm{Lip}(\overline \Omega)}:=\|\mathbf{u}_r\|_{L^\infty(\Omega)}+\| \nabla \mathbf{u}_r\|_{L^\infty(\Omega)} \leq C,
	\]
for any $0< r < \bar r$, and any minimizer $\mathbf{u}_r$ of $c_r$. \end{theorem}

Observe that, for each $r>0$ fixed, Lipschitz regularity is proved via a barrier argument, which is possible due to the exterior sphere condition (see \cite[Theorem 3.4]{STTZ2018}). However the barrier used depends on the radius, and the argument breaks down as $r\to 0^+$. Here we rely on different methods.

Combining this theorem with the information obtained in previous papers about the local case $r=0$, we have the following. Finer results for the singular set are proved in the recent paper \cite{Alp}.

\begin{corollary}\label{coro:unif Lip}
There exists $C>0$ such that
\[
	c_0 \leq c_r \leq c_0 + C r \quad \text{ for sufficiently small $r>0$.}
	\]
In particular, $c_r\to c_0$ as $r\to 0$. Moreover, given any minimizer $\mathbf{u}_r$ of $c_r$ for $r>0$,  there exists $\mathbf{u}_0\in H^1_0(\Omega)\cap \mathrm{Lip}(\overline \Omega)$, solution to $c_0$, such that, up to a subsequence, 
\[
\mathbf{u}_r\to \mathbf{u}_0 \quad \text{ strongly in } H^1_0(\Omega)\cap C^{0,\alpha}(\overline \Omega), \text{ for every } \alpha\in (0,1).
\]

\end{corollary}

We believe that these results may pave the way to the development of a common free boundary regularity theory. This will be the object of future investigations.

\medbreak

A closely related problem concerns the regularity of singularly perturbed harmonic maps and of their free boundaries. Under the previous assumptions on $\Omega$, let 
\[
\Omega_{\bar r}=\bigcup_{x \in \Omega} B_{\bar r}(x)=\{x\in \R^N:\ \dist (x,\Omega)<\bar r\},
\] 
and, given $k \ge 2$ nonnegative nontrivial functions $f_1,\ldots, f_k\in H^1(\Omega_{\bar r})\cap C(\overline{\Omega_{\bar r}})$ satisfying
\[
\dist(\supp f_i,\supp f_j)\ge \bar r\quad \forall i\neq j, \quad \supp f_i \cap (\Omega_{\bar r} \setminus \Omega) \neq \emptyset \quad \forall i,
\]
let us consider the minimization problems 
\[
h_r:= \inf_{\mf{u}\in H_r} \sum_{i=1}^k \int_\Omega |\nabla u_i|^2, \quad r \in [0,\bar r),
\] 
where 
\begin{equation}\label{def H infty}
H_r=\left\{\mf{u} = (u_1,\dots,u_k) \in H^1(\Omega_{\bar r},\R^k)\left| \begin{array}{l} \int_{\Omega} u_i^2 u_j^2 = 0 \text{ and } \dist(\supp u_i,\supp u_j)\ge r\quad  \forall i\neq j \\ 
u_i = f_i \ \text{a.e. in $\Omega_{\bar r} \setminus \Omega$}
\end{array}\right. \right\}.
\end{equation}
As for the optimal partition problems, the local case $r=0$ is essentially understood (see \cite{CafLin08, TT}), while for the nonlocal one $r>0$, studied in \cite{STTZ2018}, there are still many open questions. However, local and nonlocal cases share the same optimal regularity for the minimizers: if $\mf{u}_r$ is a minimizer of $h_r$, then it is locally Lipschitz continuous in $\Omega$, both for $r=0$ and $r>0$. Therefore, it is natural to wonder whether a result similar to Theorem \ref{thm:unif Lip} holds true or not. We can give an affirmative answer.

\begin{theorem}\label{thm:unif Lip 2}
For any compact set $K \subset \subset \Omega$, there exists a constant $C > 0$ (which depends on $K$, $\Omega$, $N$ and on $\bar r$) such that 
 \[
	\|\mathbf{u}_r\|_{\mathrm{Lip}(K)}:=\|\mathbf{u}_r\|_{L^\infty(K)}+\| \nabla \mathbf{u}_r\|_{L^\infty(K)} \leq C,
	\]
for any $0< r < \bar r$, and any minimizer $\mathbf{u}_r$ of $h_r$. Moreover there exists $\mathbf{u}_0\in H_0\cap \mathrm{Lip}_{\loc}(\Omega)$, solution to $h_0$, such that, up to a subsequence, 
\[
\mathbf{u}_r\to \mathbf{u}_0 \quad \text{ strongly in } H^1_{\loc}(\Omega)\cap C^{0,\alpha}_{\loc}(\Omega), \text{ for every } \alpha\in (0,1).
\]
\end{theorem}

Problems $c_r$ and $h_r$ are closely related, both for $r=0$ and $r>0$. In turn, they are both related to the study of the asymptotic behavior of multi-components system in the limit of strong competition. This topic attracted a lot of attention in the last decades, mainly in the local setting, for which by now a variety of results are available: systems with symmetric quadratic interaction between the different densities were studied in \cite{CKL, ctv, CTVind, SoZi}; systems with variational cubic interaction in \cite{CafLin08, CTV2002, CTV2003, DaWaZh, NTTV1, STTZ16, SoZi, SoZi2, WeiWeth}; analogue problems for systems driven by the fractional Laplacian were addressed in \cite{dST, TVZ14, TVZ16, ToZi, VZ14}; the fully nonlinear setting was studied in \cite{CPQT, Q13}; and systems with asymmetric diffusion or asymmetric interaction were tackled in \cite{SoTe20, TVZ19, WeiWeth}. See also the references therein.

In contrast, besides \cite{STTZ2018}, the only contributions regarding long range interaction models are \cite{CPQ} and \cite{Boz}; in \cite{CPQ}, the authors analyzed the spatial segregation for systems such as
\begin{equation}\label{syst p}
\begin{cases}
\Delta u_{i,\beta} = \beta u_{i,\beta}\sum_{j \neq i} (\chi_{B_1} \star |u_j|^p)  & \text{in $\Omega$} \\
u_{i,\beta} = f_i \geq 0 & \text{in $\Omega_1 \setminus \Omega$},
\end{cases}
\end{equation}
with $1 \le p < +\infty$. In the above equation, $\chi_{B_1}$ denotes the characteristic function of $B_1(0)$, and $\star$ stays for the convolution. The authors proved the equi-continuity and gradient bounds for families of viscosity solutions $\{\mf{u}_\beta: \beta > 0\}$ to \eqref{syst p}, the local uniform convergence to a limit configuration $\mf{u}$, and then studied the free-boundary regularity of the positivity sets $\{u_i>0\}$ in the case $p=1$ and dimension $N=2$. In \cite{Boz}, the author proved a uniqueness result.

\subsection*{Notation and structure of the paper}

We mainly use standard notation. Whenever a function $f$ is radially symmetric, we write $f(x)=f(|x|)$. We denote by $B_r(x_0)$ the Euclidean ball of radius $r>0$ and center $x_0$; whenever $x_0=0$, we simply write $B_r$. In most of the integrals, the volume or surface elements are omitted, for the sake of brevity; the domain of integration suggests the natural choice. 

The rest of the paper is devoted to the proof of Theorem \ref{thm:unif Lip}. We focus on the case $N \ge 3$. In Section \ref{sec: pre}, we present some preliminary inequalities regarding eigenfunctions of the Laplacian. Section \ref{sec: main} contains the proof of Theorem \ref{thm:unif Lip}. Concerning the case $N=2$ in Theorem \ref{thm:unif Lip}, and Theorem \ref{thm:unif Lip 2}, we shall not present the details. The proof follows the same sketch of the one of Theorem \ref{thm:unif Lip}, being actually a bit simpler at several points. We will stress the main differences in some remarks whenever necessary.

\section{Preliminary results}\label{sec: pre}
We devote this section to some inequalities about eigenfunctions of the Laplacian that will be crucial in order to reach the conclusion of Theorem \ref{thm:unif Lip}. Some of these inequalities are already known and are presented here for the sake of clarity. Some others may be of independent interest and are given in a general setting.

\subsection{Pointwise estimate of the gradient of eigenfunctions} We show that the maximum of the gradient of a positive eigenfunction is reached at the boundary of its domain, up to a multiplicative constant depending only on the dimension, and in particular not on the domain $\Omega$. The following result can be extended to more general bounded domains (in which case the gradient may be unbounded), but we state and prove it only under an additional regularity assumption on $\Omega$.
 
\begin{lemma}\label{lem max bound}
	Let $\Omega \subset \R^N$ be a nonempty bounded domain that enjoys the exterior sphere condition (of any radius) at any point of its boundary. Let $\lambda = \lambda_1(\Omega)$ be the first positive eigenvalue of the Laplacian with Dirichlet boundary conditions with eigenfunction $u \in H^1_0(\Omega)$,
	\[
	\begin{cases}
		-\Delta u = \lambda u &\text{in $\Omega$}\\
		u = 0&\text{on $\partial \Omega$}.
	\end{cases}
	\]
	There exists a universal constant $C = C(N) > 0$ and a sequence $\{x_n\} \subset \Omega$ such that 
	\[
	\lim_{n \to +\infty} \dist(x_n,\partial \Omega) = 0 \quad \text{and} \quad \liminf_{n \to +\infty} |\nabla u(x_n)| \geq C \| \nabla u\|_{L^\infty(\Omega)}.
	\]
\end{lemma}
\begin{proof}
	By classical regularity theory of elliptic equations, we know that the eigenfunction $u$ is a $C^\infty$ function inside of $\Omega$ and is Lipschitz continuous up to the boundary \cite[Proposition 2.20]{HanLin}, and by the maximum principle we can assume that $u > 0$ in $\Omega$. Exploiting the regularity of $u$ inside of $\Omega$, we find that the function $x \mapsto |\nabla u(x)|$ is continuous and bounded in $\Omega$. In order to reach the conclusion, since $\Omega$ is bounded, it suffices to show that, if $|\nabla u(x)|$ attains its maximum inside of $\Omega$, then its maximum value is still comparable to the value of the gradient close to a point on the boundary. Hence we can further assume that there exists $y \in \Omega$ such that 
	\[
	\| \nabla u \|_{L^\infty(\Omega)} = |\nabla u(y)|.
	\]
	Letting $r = \dist(y, \partial \Omega) > 0$, we consider the function $v \in \mathrm{Lip}(\overline{B_1})$ defined as
	\[
	v(x) := \frac{u(y + r x)}{r |\nabla u(y)|}.
	\]
	Then, by definition, we see that $v > 0$ and $|\nabla v| \leq 1$ in $B_1$,  with $|\nabla v(0)| = 1$ and
	\begin{equation}\label{eq eigen scale v}
		\begin{cases}
			-\Delta v = \lambda r^2 v &\text{in $B_1$}\\
			v(z) = 0&\text{for some $z \in \partial B_1\cap \frac{\partial \Omega-y}{r}$}.
		\end{cases}
	\end{equation}
	Observe that, by set inclusion, we find $\lambda r^2 \leq \lambda_1(B_1)$, the first Dirichlet eigenvalue of the unit ball in $\R^N$. We want to show that $v(0) \geq m$ for some $m > 0$ that depends only on the dimension $N$. By elliptic regularity theory \cite[Corollary 6.3]{GT}, we know that there exists a constant $C_N > 0$ that depends only on the dimension $N$ such that
	\[
	\|D^2 v\|_{L^\infty(B_{1/2})} \leq C_N \left( \|v\|_{L^\infty(B_{1})}  + \| \lambda r^2 v\|_{L^\infty(B_{1})}+\|\lambda r^2 \nabla v\|_{L^\infty(B_{1})} \right) \leq 2C_N \left(1+ \lambda_1(B) \right),
	\]
	where $D^2 v$ is the Hessian matrix of $v$. Let 
	\[
		A_N = \max\left(3,  2C_N \left(1+ \lambda_1(B_1)\right) \right)
	\]
which, ultimately, depends only on the dimension $N$. For any $x \in B_{1/2}$ we have
	\[
		v(x) = v(0) + \nabla v(0) \cdot x + R(x)
	\]
	where the remainder verifies $|R(x)| \leq A_N \|x\|^2 / 2$. We now take
	\[
		x_0 = - \frac{1}{A_N} \nabla v(0),
	\]
which belongs to $B_{1/2}$ since $A_N>2$ and $|\nabla v(0)|=1$. Recalling that $v > 0$ in $B_1$ and using again the fact that $|\nabla v(0)|=1$, we find
	\[
	0 \leq v(x_0) \leq v(0) - \frac{1}{A_N}|\nabla v(0)|^2+\frac{1}{2A_N} |\nabla v(0)|^2=v(0)-\frac{1}{2A_N},
	\]
	that is
	\[
		v(0) \geq \frac{1}{2A_N} > 0.
	\] 
	Combining this estimate with the fact that $|\nabla v| \leq 1$ we have that
	\[
		\min \left\{ v(x) : |x| \leq \frac{1}{4A_N} \right\} \geq \frac{1}{2A_N} - \frac{1}{4 A_N} = \frac{1}{4A_N} > 0.
	\]	
	We now consider the function $\underline v \in C^2(B_1 \setminus B_{1/(4A_N)})$ defined as
	\[
		\underline{v}(x) = D_N \left(\frac{1}{|x|^{N-2}} - 1\right),
	\]
	for a constant $D_N>0$ defined by the relation $D_N((4A_N)^{N-2}-1)=1/(4A_N)$. Therefore, $\underline{v}$ is the solution to the problem 
	\[
	\begin{cases}
	-\Delta \underline{v} = 0 \leq -\Delta v &\text{in $B_1\setminus B_{1/(4A_N)}$}\\
	\underline{v} = 0 \leq v&\text{on $\partial B_1$}\\
	\underline{v} = \frac{1}{4A_N} \leq v&\text{on $\partial B_{1/(4A_N)}$}.
	\end{cases}
	\]
	Notice that $\underline{v}$ is radially decreasing, $\partial_r \underline{v}$ is radially increasing, and 	\[
	\partial_r \underline{v}(x) \leq \partial_r \underline{v}(z) = (2-N) D_N =: -\kappa_N < 0 \quad \forall x \in B_1\setminus B_{1/(4A_N)}.
	\]
	Moreover, by the maximum principle, $\underline{v} \leq v$ in $B_1 \setminus B_{1/(4A_N)}$. We claim that this implies that there exists a sequence $\{z_n\} \subset B_1$ such that
	\[
	z_n \to z \qquad \text{and} \qquad \liminf_{n \to \infty} |\nabla v(z_n)| \geq \kappa_N.
	\]
	Indeed, let us assume by contradiction that there exists $\eps > 0$ such that for any $x \in B_\eps(z) \cap B_1$ we have $|\nabla v(x) |< \kappa_N$. We consider the function $f \in \mathrm{Lip}([0,1])$, defined as $f(t) = v((1-t)z)$ for all $t \in [0,1]$. We have that $|f'(t)| = |\nabla v((1-t)z) \cdot z  | < \kappa_N$ for all $t \in (0,\eps)$, thus
	\[
	f(\eps) = f(0) + \int_{0}^{\eps} f'(s) ds \leq \int_{0}^{\eps} |f'(s)| ds < \eps \kappa_N \implies v((1-\eps)z) < \eps \kappa_N.
	\]
	On the other hand, by the same reasoning as before we have that
	\[
	\underline v((1-\eps)z) = -\int_{0}^{\eps}  \partial_r \underline v((1-s)z) ds \geq \eps \kappa_N,
	\]
	in contradiction with the fact that $\underline{v} \leq v$ in $B_1 \setminus B_{1/(4A_N)}$. The conclusion follows by scaling back to the original function $u$.
\end{proof}

\subsection{Mean-value property for eigenfunctions} We show that  the eigenfunctions of the Laplacian and their gradients enjoy a mean-value property similar to harmonic functions. For a given $\bar \lambda>0$, let $\bar R = \bar R (\bar \lambda)> 0$ be such that the ball $B_{2\bar R}$ has first Dirichlet eigenvalue equal to $\bar \lambda$. We denote by $\varphi$ the corresponding positive eigenfunction, with
\begin{equation}\label{eq:eigenfunction_def}
\begin{cases}
	-\Delta \varphi = \bar \lambda \varphi & \text{in $B_{2 \bar R}$},\\
	\varphi = 0 &  \text{on $\partial B_{2 \bar R}$},\\
	\varphi(0) = 1.
\end{cases}
\end{equation}
We recall that $\varphi$ is radially symmetric and radially decreasing, attaining its only maximum at the origin and $\varphi(x) = J_{N/2-1,1}(\alpha|x|)$ where $J_{N/2-1,1}$ is the Bessel function of first kind and index $N/2-1$, and $\alpha > 0$ is a suitable scaling parameter. 

\begin{lemma}\label{lem mean prop}
	Let $R \leq \bar R$ and assume there exists a  nonnegative function $v \in C^\infty(B_R)$ such that
	\[
	-\Delta v \leq \lambda v \qquad \text{in $B_R$},
	\]
	for $\lambda \leq \bar \lambda$. Then for any $r \in (0,R)$ we have
	\[
	\frac{1}{r^N} \int_{B_r} v \leq  \frac{1}{\varphi( R) R^N} \int_{B_R} v.
	\]
\end{lemma}
\begin{proof}
	First we observe that, since $\varphi>0$ in $B_R$, 
	\begin{equation*}
		-\div \left( \varphi^2 \nabla \left(\frac{v}{\varphi}\right)\right) = - \Delta v \varphi + \Delta \varphi v  \le (\lambda-\bar \lambda)v \varphi  \le 0 \quad \text{ in } B_R.
	\end{equation*}
	For $0<r<R$, integrating the previous inequality in $B_r$ we find
	\[
	0 \leq \int_{B_r} \div \left( \varphi^2 \nabla \left(\frac{v}{\varphi}\right)\right)  =  \varphi^2(r)  \int_{\partial B_r}  \partial_\nu \left(\frac{v}{\varphi}\right) \implies \int_{\partial B_r}  \partial_\nu \left(\frac{v}{\varphi}\right) \geq 0.
	\]
Introduce the smooth function $\Phi:(0,R)\to \R$ as
	\[
	\Phi(r) := \frac{1}{r^{N-1}} \int_{\partial B_r} \frac{v(y)}{\varphi(y)}\,d\sigma_y = \int_{\partial B_1}\frac{v(r x)}{\varphi(rx)}\,d\sigma_x.
	\]
	Taking the derivative of $\Phi$ yields
	\[
	\Phi'(r) = \frac{1}{r^{N-1}} \int_{\partial B_r} \partial_\nu \left(\frac{v}{\varphi}\right) \geq 0,
	\]
	that is,  the function $r\mapsto \Phi(r)$ is positive and increasing for $r<R$. As a result, for any $0<s<t<R$ we have
	\[
	t^{N-1} \int_{\pa B_s} \frac{v}{\varphi}  \leq  s^{N-1} \int_{\pa B_t} \frac{v}{\varphi}.
	\]
	Next, for a given $r\in (0,R)$, we integrate the previous inequality for $s \in (0,r)$ and afterwards for $t \in (r,R)$, and  deduce that
	\[
	\left(\frac{R^N}{N}- \frac{r^N}{N}\right)\int_{B_r} \frac{v}{\varphi} \leq \frac{r^N}{N}\int_{B_R\setminus B_r} \frac{v}{\varphi}.
	\]
	By rearranging the terms we obtain
	\[
	\frac{1}{r^N} \int_{B_r} \frac{v}{\varphi} \leq \frac{1}{R^N} \int_{B_R} \frac{v}{\varphi}.
	\]
	To conclude we recall that $\varphi$ is decreasing in $r$ and that $\varphi(0)=1$.
\end{proof}

A direct consequence of the mean-value property is a similar inequality for the gradient of eigenfunctions.
\begin{corollary}\label{cor mean}
	Let $R \leq \bar R$ and assume there exists a  function $u \in C^\infty(B_R)$ such that
	\[
	-\Delta u = \lambda u \qquad \text{in $B_R$},
	\]
	for $2\lambda \leq \bar \lambda$. Then for any $r \in (0,R)$ we have
	\[
	\frac{1}{r^N} \int_{B_r} |\nabla u|^2 \leq  \frac{1}{\varphi( R) R^N} \int_{B_R} |\nabla u|^2
	\]
	and, in particular,
	\[
	|\nabla u(0)|^2\leq \frac{1}{\varphi(R) |B_R|} \int_{B_R} |\nabla u|^2.
	\]
\end{corollary}
\begin{proof}
	It suffices to consider Lemma \ref{lem mean prop} with $v = |\nabla u|^2$, since
	\[
	-\Delta |\nabla u|^2 = 2  \left(\lambda |\nabla u|^2 - \sum_{i=1}^N |\nabla u_{x_i}|^2 \right) \leq 2 \lambda |\nabla u|^2 \quad \text{ in } B_R
	\]
	and $2\lambda \leq \bar \lambda$.
\end{proof}

\subsection{Energy estimate of the gradient of eigenfunction} Previously we have shown a mean-value property for the gradient of eigenfunction in the interior of their support. In this section we prove a similar result for points on the boundary. It rests on a monotonicity formula of Alt-Cafferelli-Friedman type for a single function defined in a domain that enjoys the exterior sphere condition. We thus first prove such formula.

As before, we fix $\bar \lambda>0$ and let $\bar R = \bar R (\bar \lambda)> 0$ be such that the ball $B_{2\bar R}$ has first Dirichlet eigenvalue equal to $\bar \lambda$, with eigenfunction $\varphi$ normalized in such a way that $\varphi(0)=1$.  Let $\Gamma_\varphi \in C^2(B_{3\bar R/2}\setminus\{0\})$ be a positive and radial solution of
\begin{equation}\label{eq:fundamental_eq}
	-\div\left(\varphi^2 \nabla \Gamma_\varphi\right) = \delta \qquad \text{in $B_{3\bar R/2}$},
\end{equation}
where $\delta$ is the Dirac delta centered at the origin. A direct computation shows that we can choose
\[
	\Gamma_\varphi(r) = (N-2) \int_r^{3 \bar R/2} \frac{s^{1-N}}{\varphi^2(s)} ds,\qquad r=|x|.
\]
With this choice we additionally have that $\Gamma_\varphi(3\bar R/2)=0$, $\Gamma_\varphi(r) > 0$ for any $r\in(0,3\bar R/2)$, and $\Gamma_\varphi'(r)=-\frac{N-2}{\varphi^2(r)r^{N-1}}$. We also define 
\begin{equation}\label{eq:def_psi}
\psi(r):=r^{N-2}\varphi^2(r)\Gamma_\varphi(r).
\end{equation}
which we assume to be extended by continuity for $r = 0$. We have the following.
\begin{lemma}\label{lem psi}
	 The function $\psi$ is Lipschitz continuous in $B_{3\bar R/2}$ and radially symmetric. For any $r \in [0,\bar R]$, $\psi(r) > 0$, while $\psi(3\bar R/2) = 0$ and there exists $C=C(N, \bar \lambda) \ge 0$ such that
	\begin{equation}\label{eq:psi}
		|\psi(r)-1| \le C r \qquad \text{for $r\in (0,3\bar R/2)$}.  
	\end{equation}
\end{lemma}
\begin{proof}
	We only need to show \eqref{eq:psi}, as the other properties in the statement are direct consequences of the definition of the function $\psi$. We have
	\begin{align*}
		\frac{\psi(r)-1}{r}&=\frac{1}{r} \left((N-2) \int_r^{3\bar R/2} \frac{s^{1-N}}{r^{2-N}}\frac{\varphi^2(r)}{\varphi^2(s)}\, ds -1\right) \\
		&=\frac{1}{r} \left((N-2) \int_r^{3\bar R/2} \frac{s^{1-N}}{r^{2-N}}\frac{\varphi^2(r)}{\varphi^2(s)}\, ds -(N-2)\int_r^{+\infty} \frac{s^{1-N}}{r^{2-N}}\, ds\right)\\
		&=(N-2)r^{N-3}\left( \int_r^{3\bar R/2} s^{1-N} \left( \frac{\varphi^2(r)}{\varphi^2(s)}-1 \right)\, ds - \int_{3\bar R/2}^{+\infty} s^{1-N}\, ds \right)\\
		&=(N-2)r^{N-3}\left( \int_r^{3\bar R/2} s^{1-N} \frac{\varphi(r)+\varphi(s)}{\varphi^2(s)}(\varphi(r)-\varphi(s))  \, ds - \frac{(3\bar R)^{2-N}}{(N-2) 2^{2-N}} \right).
	\end{align*}
	Now observe that, by monotonicity of $\varphi$,
	\[
	0\leq \frac{\varphi(r)+\varphi(s)}{\varphi^2(s)} \leq \frac{2\varphi(0)}{\varphi^2(3\bar R/2)}=\frac{2}{\varphi^2(3\bar R/2)},
	\]
	and that, since $\varphi$ is smooth and radial, 
	\[
	\varphi(r)=1+\varphi''(\xi)\frac{r^2}{2},\qquad \varphi(s)=1+\varphi ''(\eta)\frac{s^2}{2}
	\]
	for some $\xi,\eta\in (0,3\bar R/2)$. Therefore, there exists $C>0$, depending on $\bar R$, such that
	\[
	\left|\frac{\psi(r)-1}{r}\right|\leq Cr^{N-3} \left(r^2\int_r^{3\bar R/2} s^{1-N}\, ds + \int_r^{3\bar R/2} s^{3-N}\, ds   + C\right)\leq C,
	\]
	since $N\geq 3$.
\end{proof}

We are now in a position to state the monotonicity formula. We work with the family of open domains
\[
B_r\setminus \overline{B_1(-e_1)}=\left\{ x \in B_{r} : (x_1 + 1)^ 2 + \sum_{i = 2}^N x_i^2 > 1\right\},
\]
where $e_1=(1,0,\ldots, 0)$ is the first vector of the canonical basis of $\R^N$.
\begin{proposition}\label{lem acf one eig}
Let $\lambda\leq \bar \lambda$ and let $u \in H^1(B_{\bar R})$ be a nonnegative solution to
\[
	\begin{cases}
	-\Delta u \leq \lambda u &\text{in $B_{\bar R}\setminus \overline{B_1(-e_1)}$}\\
	u = 0 &\text{in } B_{\bar R} \cap \overline{B_1(-e_1)}.
	\end{cases}
\]
Then there exist $C = C(N,\bar \lambda)> 0$ and $\tilde r = \tilde r(N,\bar \lambda)>0$, such that
the function
\begin{equation}\label{eq:Psi}
	\Psi(r) := e^{ C r } \frac{1}{r^2} \int_{B_r} \frac{\psi}{|x|^{N-2}} \left|\nabla \left(\frac{u}{\varphi} \right)\right|^2=e^{ C r } \frac{1}{r^2} \int_{B_r} \varphi^2 \Gamma_\varphi\left|\nabla \left(\frac{u}{\varphi} \right)\right|^2.
\end{equation}
	is nondecreasing in $r \in (0,\tilde r)$, and
	\begin{equation}\label{eq:Psi_H^1_estimate2}
		\frac{1}{r^N}\int_{B_r} |\nabla u|^2 \leq C\Psi(r) \qquad \forall r \in (0,\tilde r).
	\end{equation}
Moreover, if $-\Delta u=\lambda u$  in $\{u>0\}$, and $\{u=0\}$ has locally finite perimeter,
then:
	\begin{equation}\label{eq:Psi_H^1_estimate}
		\Psi(r) \le \frac{C}{r^N} \int_{ B_r} |\nabla u|^2 \qquad r\in (0,\tilde r)
	\end{equation}
\end{proposition}

We start by stating and proving an estimate of the first eigenvalue of spherical caps, and a Poincar\'e-type inequality.

\begin{remark}
	In dimension $N=2$ we need to change the definition of function $\Psi$ in \eqref{eq:Psi} as follows:
	\[
		\Psi(r) = e^{ C r } \frac{1}{r^2} \int_{B_r} \varphi^2 \left|\nabla \left(\frac{u}{\varphi} \right)\right|^2.
	\]
	The proof follows by similar computations. 
\end{remark}

\begin{lemma}[Estimates for eigenvalues]\label{lem spher cap}
Consider the spherical caps
\[
	\omega_r := \partial B_1 \setminus \overline{B_{1/r}(-e_1/r)} = \left\{ y\in \partial B_1: \left(y_1+\frac{1}{r}\right)^2+\sum_{i=2}^N y_i^2 >\frac{1}{r^2}\right\} = \left\{ y \in \partial B_1 : y_1 > -\frac{r}{2} \right\}
\]
and let $\lambda_1(\omega_r)$ stand for the first Dirichlet eigenvalue of the Laplace-Beltrami operator on $\omega_r$:
\[
	\lambda_1(\omega_r) = \inf_{u \in H^1_0(\omega_r)} \frac{\int_{\omega_r} |\nabla_T u|^2 }{\int_{\omega_r} |u|^2 },
\]
where $\nabla_T u$ is the tangential gradient of $u$. Then there exist $\bar r=\bar r(N)$ and  $C=C(N)>0$ such that
\begin{equation}\label{eqn eig est}
	N - 1 -Cr \leq \lambda_1(\omega_r) \leq N-1\qquad \text{for $r\in(0,\bar r)$.}
\end{equation}

\end{lemma}
\begin{proof}
The sets $\omega_r$ are invariant under rotations with respect to the first axis. As a result, the first eigenfunction depends only on $\theta=\arccos \langle y,e_1\rangle \in [0,\pi]$, the polar angle with $e_1$ (see \cite{Spe73}). We have
\begin{equation}\label{inequ eigen}
	\lambda_1(\omega_r) = \inf \left\{ \frac{\int_0^{\theta_r } (\sin \theta)^{N-2} | w'(\theta) |^2\, d\theta }{\int_0^{\theta_r} (\sin \theta)^{N-2} w^2(\theta)\, d\theta }\left| \begin{array}{l} w \in H^1\left([0,\theta_r]\right),\\ 
w\left(\theta_r\right) = 0\end{array}\right.\right\}, 
\end{equation}
where $\theta_r>0$ is
\[
	\theta_r = \arccos\left( -\frac{r}{2} \right) = \frac{\pi}2 +\frac{r}2+ O(r^3)
\]
for $r>0$ small enough. The first eigenvalue of $\omega_r$ is simple, and the corresponding eigenfunction is a multiple of the unique positive solution $w=w_r$ of
\[
\begin{cases}
-((\sin \theta)^{N-2}w')'=\lambda_1(\omega_r) (\sin \theta)^{N-2} w &\text{in } (0,\theta_r),\\
w'(0)=0,\   w(\theta_r)=0, \ w(0) = 1. 
\end{cases}
\]
A direct computation shows that when $r = 0$, that is $\theta_r = \pi/2$, we have
\[
w_0= \cos \theta \qquad \text{and} \qquad \lambda_1(\omega_0) = N-1.
\]
By set inclusion we can deduce that the function $r \mapsto \lambda_1(\omega_r)$ is monotone decreasing in $r$; moreover, as the first eigenvalue is simple, the function $r \mapsto \lambda_1(\omega_r)$ is differentiable at $r=0$. Thus the limit for $r \to 0$ exists and we have
\[
	\lambda_1(\omega_r) = N-1 -Ar + o(r)
\]
for a positive constant $A = A(N)$ that depends only on the dimension.

We can be more precise, by giving an explicit value for the constant in the Taylor expansion of $\lambda_1(\omega_r)$. To this aim, we make use of a shape derivative of the domain $\omega_r$.  We introduce the family of smooth diffeomorphisms $\Phi \in C^\infty([0,1] \times [0,\pi];[0,\pi])$, defined as 
\[
	\Phi(r,\theta)=\frac{2\theta_r}{\pi}\theta.
\]
We observe that $\Phi(0,\theta)=\theta$ (that is, $\Phi(0,\cdot )$ is the identity), while for any $r > 0$, $\Phi$ maps the set $[0,\pi/2]$ to the set $[0,\theta_r]$. Moreover we have $\frac{\partial \Phi}{\partial r}(0,\theta)=\frac{\theta}{\pi}$. Applying the theory of domain variation (see e.g. \cite[Th\'eor\`eme 5.7.1]{HePi}) we find that
\[
A =\frac{d}{dr}\lambda_1(\omega_r)|_{r=0}= - \frac{\int_{\partial \omega_0} (w_0')^2 \partial_r \Phi\left(0,\frac{\pi}{2}\right)}{\int_{ \omega_0} (w_0)^2}<0. \qedhere
\]
\end{proof}

Next we state and prove an inequality of Poincaré-type for $H^1$ functions that equal to zero on a ball.

\begin{lemma}[Poincar\'e-type inequality] \label{lem poinc}
For any $R>0$, there exists a constant $C_P = C_P(N, R)$ such that
\begin{equation*}
	\frac{1}{r}\int_{\partial B_r} u^2 + \frac{1}{r^2} \int_{B_r} u^2  \leq C_P  \int_{B_r} |\nabla u|^2
\end{equation*}
for any $r \in (0,R]$ and $u \in H^1(B_r)$ with $u =0$ in $B_r \cap \overline{B_1(-e_1)}$. 
\end{lemma}
\begin{proof}
We start with a change of variable, letting $v(x) = u(rx)$ we find that the statement of the result is equivalent to showing that
\[
	\int_{\partial B_1} v^2 + \int_{B_1} v^2  \leq C  \int_{B_1} |\nabla v|^2
\]
for any  $v \in H^1(B_1)$ with $v =0$ in $B_1 \cap \overline{B_{1/r}(-e_1/r)}$, $r \in (0,R)$. Assume, by contradiction, that there exist sequences $\{v_n\}\subset H^1(B)$ and $r_n \to \tilde r\in [0,R]$ such that $v_n = 0$ on $B_1 \cap \overline{B_{1/r_n}(-e_1/r_n)}$, and
\[
	\int_{\partial B_1} v_n^2 + \int_{B_1} v_n^2  = 1 \quad \text{while} \quad \int_{B_1} |\nabla v_n|^2 \to 0.
\]
We conclude that the sequence $\{v_n\}$ converges in $H^1(B_1)$ to a non-zero constant function $v \in H^1(B_1)$. On the other hand, by taking the limit of the sequence of sets $\{B_1 \cap \overline{B_{1/r_n}(-e_1/r_n)}\}$, it must be that $v = 0$ in $B_1 \cap \overline{B_{1/\tilde r}(-e_1/\tilde r)}$ if $\tilde r>0$, or $v = 0$ in $B_1 \setminus \{ x_1 \leq 0\}$ ir $\tilde r=0$, a contradiction.
\end{proof}

We state and prove a useful consequence of the previous inequality.

\begin{corollary}\label{cor poin}
There exist $C = C(N, \bar \lambda)$ and $\tilde r = \tilde r(N, \bar \lambda)>0$, such that
\begin{equation*}
\int_{B_r} |\nabla u|^2  \leq C \int_{B_r} \psi^2 \left|\nabla \left(\frac{u}{\varphi} \right)\right|^2
\end{equation*}
for any $r \in (0,\tilde r)$ and $u \in H^1(B_{r})$, with $u =0$ in $B_r \cap \overline{B_1(-e_1)}$.
\end{corollary}
\begin{proof}
	The result follows by a chain of straightforward inequalities. We have, for $r\in (0,\bar R]$,
	\[\begin{split}
		\int_{B_r} |\nabla u|^2 &= \int_{B_r} \left|\nabla u - u \frac{\nabla \varphi}{\varphi} +  u \frac{\nabla \varphi}{\varphi}\right|^2 \leq 2\int_{B_r} \left|\nabla u - u \frac{\nabla \varphi}{\varphi} \right|^2 + 2 \int_{B_r} \left|u \frac{\nabla \varphi}{\varphi}\right|^2 \\
		 &\leq 2\int_{B_r} \varphi^2 \left|\nabla \left(\frac{u}{\varphi}\right) \right|^2 + 2\left\|\frac{\nabla \varphi}{\varphi}\right\|_{L^\infty(B_r)}^2\int_{B_r} u^2\\
		 &\leq 2C(\bar R)\int_{B_r} \psi^2 \left|\nabla \left(\frac{u}{\varphi}\right) \right|^2 + 2C_P(N,\bar R) r^2 \left\|\frac{\nabla \varphi}{\varphi}\right\|_{L^\infty(B_{\bar R})}^2 \int_{B_r} |\nabla u|^2,
	\end{split}\]
where we used Lemma \ref{lem psi} and Lemma \ref{lem poinc}. The result follows by rearranging the terms in the last inequality and choosing $\tilde r=\tilde r(N,\bar R)=\tilde r(N,\bar \lambda) > 0$ sufficiently small in such a way that
\[
	2C_P{\tilde r}^2 \left\|\frac{\nabla \varphi}{\varphi}\right\|_{L^\infty(B_{\bar R})}^2 \leq \frac12. \qedhere
\]
\end{proof}

\begin{proof}[Proof of Proposition \ref{lem acf one eig}] We start by showing the monotonicity of the function $\Psi$. First of all, let $w:= u/\varphi \in H^1(B_{\bar R})$, which satisfies
\[
	\begin{cases}
		-\div(\varphi^2 \nabla w) \leq 0 &\text{in $B_{\bar R}$}\\
		w = 0 &\text{in } B_{\bar R}\cap B_{1}(-e_1).
	\end{cases}
\]
We now show that in this case there exists $C > 0$ such that
	$\Psi(r)$ defined in \eqref{eq:Psi} is monotone nondecreasing in $r$, for $ r$ sufficiently small. To start with, by formally testing the equation for $w$ by $\Gamma_\varphi w$ and integrating in $B_r$, we see that
\begin{align}
	\int_{B_r} \varphi^2\Gamma_\varphi |\nabla w|^2  &\leq  \int_{\partial B_r} \varphi^2 \Gamma_\varphi w (\partial_\nu w) - \int_{B_r} \varphi^2 w \nabla w \cdot \nabla \Gamma_\varphi \nonumber\\
									&= \int_{\partial B_r}\varphi^2 \Gamma_\varphi w (\partial_\nu w)- \int_{B_r} \varphi^2 \nabla(\frac{w^2}{2})\cdot \nabla \Gamma_\varphi \label{eq:acf_aux1}
\end{align}
(to justify rigorously this computation, it is enough to take a sequence of mollifiers $\{\rho_m\}$, work with the regular function $\rho_m\ast u\to u$ and $w_m:=(\rho_m\ast u)/\varphi$, integrate by parts in the domain $B_r\setminus B_\eps$ and let first $\eps\to 0$ and then $m\to \infty$).
Now, by testing the equation for $\Gamma_\varphi$ - \eqref{eq:fundamental_eq} - by $w^2/2$, and integrating by parts, we find
\begin{equation}\label{eq:acf_aux2}
\int_{B_r} \varphi^2 \nabla(\frac{w^2}{2})\cdot \nabla \Gamma_\varphi=\int_{\partial B_r} \varphi^2 (\partial_\nu \Gamma_\varphi) \frac{w^2}{2} + \frac{w^2(0)}{2} \geq \int_{\partial B_r} \varphi^2 (\partial_\nu \Gamma_\varphi) \frac{w^2}{2}. 
\end{equation}
By plugging \eqref{eq:acf_aux2} into \eqref{eq:acf_aux1} and recalling the definition of $\Gamma_\varphi$ and $\psi$:
\begin{align}
	\int_{B_r} \frac{\psi}{|x|^{N-2}} |\nabla w|^2 &=	\int_{B_r} \varphi^2\Gamma_\varphi |\nabla w|^2\leq \int_{\partial B_r}\left( \varphi^2 \Gamma_\varphi w (\partial_\nu w)- \frac{1}{2} w^2 \varphi^2 (\partial_\nu \Gamma_\varphi) \right) \nonumber \\
									&= \int_{\partial B_r} \left(\varphi^2 \Gamma_\varphi w (\partial_\nu w) + \frac{N-2}{2r^{N-1}} w^2\right)= \int_{\partial B_r} \left(\frac{\psi}{|x|^{N-2}} w (\partial_\nu w) + \frac{N-2}{2|x|^{N-1}} w^2\right)   \label{eq:aux_for_later}.
\end{align}
We now compute the logarithmic derivative of $\Psi$ and find
\begin{align*}
	\frac{d}{dr} \log \Psi(r) &= C - \frac{2}{r} + \frac{\displaystyle \int_{\partial B_r} \frac{\psi(x)}{|x|^{N-2}} |\nabla w|^2}{\displaystyle\int_{B_r} \frac{\psi(x)}{|x|^{N-2}} |\nabla w|^2} \geq C - \frac{2}{r} + \frac{\displaystyle \int_{\partial B_r} \frac{\psi(x)}{|x|^{N-2}} |\nabla w|^2}{\displaystyle \int_{\partial B_r} \left(\frac{\psi}{|x|^{N-2}} w (\partial_\nu w) + \frac{N-2}{2|x|^{N-1}} w^2\right) }\\
		&=C-\frac{2}{r} + \frac{\displaystyle  \frac{\psi(r)}{r^{N-2}}\int_{\partial B_r}  |\nabla w|^2}{\displaystyle \frac{\psi(r)}{r^{N-2}}\int_{\partial B_r} w (\partial_\nu w) + \frac{N-2}{2r^{N-1}} \int_{\partial B_r} w^2 }
\end{align*}
Let $v=v^{(r)} := w(rx)$, which by assumption vanishes in the complementary of the set 
\[
\omega_r = \partial B_1 \setminus \overline{B_{1/r}(-e_1/r)} \subset \partial B_1.
\]
Then
\begin{align*}
	\frac{d}{dr} \log \Psi(r) & \geq C - \frac{2}{r} + \frac{\psi(r)}{r} \frac{\displaystyle \int_{\omega_r} |\nabla v|^2}{\displaystyle \int_{\omega_r} \left(\psi(r)v(\partial_\nu v)+\frac{N-2}{2}v^2\right)}\\
					&= C - \frac{2}{r} + \frac{1}{r\psi(r)} \frac{\displaystyle \int_{\omega_r} \left(\psi^2(r)(\partial_\nu v)^2 + \psi^2(r) |\nabla_\theta v|^2\right)}{\displaystyle \int_{\omega_r}\left(\psi(r)v(\partial_\nu v)+\frac{N-2}{2}v^2\right)}\\
					&\geq C - \frac{2}{r} + \frac{1}{r\psi(r)} \frac{\displaystyle \int_{\omega_r} \left(\psi^2(r)(\partial_\nu v)^2 + \psi^2(r)\lambda_1(\omega_r) v^2 \right)}{\displaystyle \int_{\omega_r}\left(\psi(r)v(\partial_\nu v)+\frac{N-2}{2}v^2\right)}.\\
\end{align*}
Since
\begin{align*}
\int_{\omega_r}\left(\psi(r)v(\partial_\nu v)+\frac{N-2}{2}v^2\right) \leq \int_{\omega_r} \left( \frac{\psi^2(r)}{2a(N-2)}(\partial_\nu v)^2+\frac{(N-2)(a+1)}{2}v^2 \right),
\end{align*}
by choosing $a>0$ such that 
\[
\frac{1}{2a(N-2)}=\frac{(N-2)(a+1)}{2\psi^2(r)\lambda_1(\omega_r)} \iff a=\frac{1}{N-2}\sqrt{\left(\frac{N-2}{2}\right)^2+\psi^2(r)\lambda_1(\omega_r) }-\frac{1}{2} = \frac{\gamma(\lambda_1(\omega_r)  \psi^2(r))}{N-2}  ,
\] 
where
\[
\gamma(t) := \sqrt{\left(\frac{N-2}{2}\right)^2 + t} - \frac{N-2}{2},
\]
we see that
\begin{align*}
	\frac{d}{dr} \log \Psi(r) &\geq C - \frac{2}{r} + \frac{2}{r \psi(r)} \gamma \left( \psi^2(r) \lambda_1(\omega_r) \right) \\
		&= \frac{2}{r} \left(-1 + \frac{C}{2} r  + \frac{1}{\psi(r)} \gamma \left( \psi^2(r) \lambda_1(\omega_r) \right)\right).
\end{align*}
Since $\gamma(N-1)=1$ and $\gamma'(N-1)=\frac{1}{N}>0$, by Lemmas \ref{lem psi} and \ref{lem spher cap} we have the existence of constants $C_1,C_2,C_3>0$ such that
\[
	\frac{\gamma\left( \psi^2(r) \lambda_1(\omega_r) \right)}{\psi(r)} \geq \frac{\gamma((1-C_1 r)(N-1-C_2r))}{1+C_1r}\geq 1-C_3 r.
\]
for any $r$ sufficiently small. In conclusion, by choosing $C:=2C_3$, we have that $\Psi$ is nondecreasing for small $r>0$.

Next we show \eqref{eq:Psi_H^1_estimate2}, which is actually a direct consequence of Lemma \ref{lem psi} and  Corollary \ref{cor poin}. Indeed we find
\begin{equation*}
	\frac{1}{r^N}\int_{B_r} |\nabla u|^2  \leq \frac{C}{r^N} \int_{B_r} \psi^2 \left|\nabla \left(\frac{u}{\varphi} \right)\right|^2 \leq \frac{C}{r^2} \int_{B_r} \frac{\psi^2}{|x|^{N-2}} \left|\nabla \left(\frac{u}{\varphi} \right)\right|^2.
\end{equation*} 

Finally we show \eqref{eq:Psi_H^1_estimate}. Using estimate \eqref{eq:aux_for_later}, we see that
\[
\begin{split}
\Psi(r) &\leq \frac{e^{Cr}}{r^2}\int_{\partial B_r} \left(\frac{\psi}{|x|^{N-2}} \left(\frac{u}{\varphi}\right)\partial_\nu \left(\frac{u}{\varphi}\right) + \frac{N-2}{2|x|^{N-1}}\left(\frac{u}{\varphi}\right)^2\right)\\
	&= \frac{e^{Cr}\psi(r)}{r^N}\int_{\partial B_r} \left(\frac{u}{\varphi}\right)  \partial_\nu \left(\frac{u}{\varphi}\right)+ \frac{e^{Cr}(N-2)}{2r^{N+1}\varphi^2(r)}\int_{\partial B_r} u^2\\
	&\leq  \frac{e^{Cr}\psi(r)}{r^N \varphi^2(r)} \int_{\partial B_r} u\partial_\nu u -\frac{e^{Cr}\psi(r) \varphi'(r)}{r^N \varphi^3(r)}  \int_{\partial B_r} u^2 + \frac{e^{Cr}(N-2)}{2r^{N+1}\varphi^2(r)}\int_{\partial B_r} u^2 \\
	&\leq \frac{e^{Cr}\psi(r)}{r^N \varphi^2(r)} \int_{\partial B_r} u\partial_\nu u + \frac{e^{Cr}}{r^N\varphi^2(r)}\left(\frac{(N-2)}{2r} + \frac{\psi(r) |\varphi'(r)|}{\varphi(r)} \right)\int_{\partial B_r} u^2
	\end{split}
\]
Multiplying the equation $-\Delta u = \lambda u$ by $u$ and integrate by parts in $B_r\cap \{u>0\}$ (since $\{u=0\}$ has locally finite perimeter, we can apply \cite[Section 5.8 - Theorem 1]{EvansGariepy}) yields to the identity 
\[
\int_{B_r}|\nabla u|^2=\int_{\{u>0\}\cap B_r}|\nabla u|^2=\lambda \int_{\{u>0\}\cap B_r} u^2+\int_{\partial (\{u>0\}\cap B_r)} u \partial_\nu u= \lambda \int_{B_r} u^2 + \int_{\partial B_r} u \partial_\nu u
\]
which in turns give us the estimate 
\[
\int_{\partial B_r} u \partial_\nu u\leq \int_{B_r} |\nabla u|^2.
\]
By Lemma \ref{lem poinc},
\[
	\int_{\partial B_r} u^2 \leq C_P r \int_{B_r} |\nabla u|^2,
\]
and we can conclude that
\[
	\Psi(r) \leq \frac{e^{Cr}}{r^N \varphi(r)^2} \left[ \psi(r)(1+ r^2 \lambda C_P) + C_P\left(\frac{(N-2)}{2} +\frac{\psi(r) |\varphi'(r)|r}{\varphi(r)} \right) \right] \int_{B_r} |\nabla u|^2
\]
finally yielding to
\[
	\Psi(r) \leq C\frac{1}{r^N } \int_{B_r} |\nabla u|^2
\]
for any $r\in (0,\tilde r)$.
\end{proof}

We cite a useful corollary that is a straightforward consequence of Proposition \ref{lem acf one eig}.

\begin{corollary}\label{cor acf one eig}
	Let $\Omega \subset \R^N$ be a connected open (and non-empty) set that enjoys the exterior sphere condition at any point of its boundary, which we assume to have locally finite perimeter. Assume, moreover, that at $x_0 \in \partial\Omega$ the exterior sphere has radius at least equal to $r_0$. Let $\lambda = \lambda_1(\Omega)$ be the first eigenvalue of the Laplacian with Dirichlet boundary conditions, and assume that $\lambda \leq \bar \lambda$. Let $u \in H^1_0(\Omega)$ be the corresponding eigenfunction. 
	There exist $C = C(N, \bar \lambda)$ and $\tilde r=\tilde r(N,\bar \lambda)$ such that
	\[
		\frac{1}{r^N}\int_{B_r(x_0)} |\nabla u|^2 \leq C \frac{1}{R^N}\int_{B_R(x_0)} |\nabla u|^2
	\]
	for any $0 < r < R \leq r_0 \tilde r$.
\end{corollary}
\begin{proof}
By a change of variables, the problems reduces to the one where $r_0=1$. In such a case, by Proposition \ref{lem acf one eig}, we have the existence of $C,C',\tilde r$, depending only on $N$ and $\lambda$ such that, whenever $0 < r < R < \tilde r$,
\[
\frac{1}{r^N}\int_{B_r(x_0)}|\nabla u|^2\leq C\Psi(r)\leq C \Psi(R) \leq \frac{C'}{R^N}\int_{B_r} |\nabla u|^2,
\]
which concludes the proof.
\end{proof}
\section{Uniform bounds}\label{sec: main} 

In this section we prove Theorem \ref{thm:unif Lip}. Assume without loss of generality that it satisfies the uniform exterior sphere condition of radius larger than or equal to $1$. Recall that $\bar r>0$ denotes a value such that $\mathcal{P}_r(\Omega) \neq \emptyset$, for every $r\in [0,\bar r)$. In what follows, for $r\in (0,\bar r)$, we let $\mathbf{u}_r=(u_{1,r},\ldots, u_{k,r})$ be a nonnegative minimizer for $c_r$ (recall the characterization \eqref{eq:weak_characterization}), with $(\Omega_{1,r},\ldots, \Omega_{k,r}):=(\{u_{1,r}>0\},\ldots, \{u_{k,r}>0\})$ being an optimal partition. Recall also that properties (1)-(3) and (a)-(b) hold true. The main idea of the proof is to show that, if the eigenfunctions do not have uniformly bounded gradients, then it is possible to construct a competitor for the minimization problem that has a smaller energy, thus contradicting the minimality of $\mathbf{u}_r$.

The starting point is to prove uniform bounds  of the eigenfunctions in the $H^1$ and the $L^\infty$ norms.

\begin{lemma}\label{eq:uniformbounds_H^1_L^infty} There exist constants $C, \Lambda >0$ such that 
\[
\| \mathbf{u}_r\|_{H^1_0(\Omega)}, \|\mathbf{u}_r\|_{L^\infty(\Omega)} \leq C 
\]
and
\[
\lambda_1(\Omega) \le \lambda_{i,r}:= \lambda_{1}(\{u_{i,r}>0\}) \leq \Lambda \quad \forall i=1,\dots,k,
\]
for every $r\in (0,\bar r)$.
\end{lemma}
\begin{proof}
The lower bound on $\lambda_{i,r}$ follows from the monotonicity of the eigenvalues with respect to domain inclusion. On the other hand, since $r\mapsto \mathcal{P}_r(\Omega)$ is decreasing with respect to domain inclusion, then $r\mapsto c_r$ is monotone increasing and, in particular, $c_r\leq c_{\bar r}$ for $0\leq r<\bar r$ and
\[
\sum_{i=1}^k \int_{\Omega} |\nabla u_{i,r}|^2=\sum_{i=1}^k \lambda_{1}(\{u_{i,r}>0\})  \leq c_{\bar r}.
\]
Since $u_{i,r}\in H^1_0(\Omega)$ is a positive solution to $-\Delta u_{i,r}\leq \lambda_{1}(\{u_{i,r}>0\}) u_{i,r}$ in $\Omega$, the $L^\infty$-uniform bounds are a standard consequence of the Brezis-Kato iteration technique (see for instance the proof of Corollary 1.6 in \cite{NTTV2} for the precise details in this framework).\end{proof}

We assume from now on, by virtue of a contradiction argument, that the gradient of $\mathbf{u}_r$ is \emph{not} uniformly bounded. That is, there exist a sequence  $\{r_n\}\subset (0,\bar r)$, a sequence of minimizers $\{\mathbf{u}_n\}$ associated with $c_{r_n}$, and a sequence of indexes $\{i_n\}$ such that 
\begin{equation}\label{eq:contradiction_assumption}
	M_n :=  \max_{i=1,\dots,k} \|\nabla u_{i,n} \|_{L^\infty(\Omega)} = \|\nabla u_{i_n,n} \|_{L^\infty(\Omega)} \to +\infty \qquad \text{as $n \to +\infty$}.
\end{equation}
Up to a subsequence and a relabelling, we can suppose that $i_n=1$ for every $n$. In what follows, we work constantly under this assumption. 

\medskip
\noindent\textbf{Notation.} In what follows we take $\bar \lambda$, the constant appearing in Section \ref{sec: pre}, equal to $\Lambda$, the upper bound of the eigenvalues $\lambda_{i,r}$ (see Lemma \ref{eq:uniformbounds_H^1_L^infty}). Moreover, without loss of generality, we assume $\tilde r = 1$ in Corollary \ref{cor acf one eig}.

\begin{lemma}
We have $r_n\to 0$.
\end{lemma}
\begin{proof}
Assume that the thesis is false. Then, up to striking out a subsequence, we have that $r_n \to r_0$ for some $r_0>0$. We recall that each $u_{i,n} \in H^1_0(\Omega_{i,r_n})$ solves $-\Delta u_{i,n}=\lambda_{i,r_n} u_{i,n}$ in $\Omega_{i,r_n}$. All of these sets satisfy a  $\frac12 r_0$-uniform exterior sphere condition, for any $n$ sufficiently large. By \cite[Theorem 3.4]{STTZ2018} we have that there exists a constant $C>0$ such that
\[
\|\nabla u_{i,n}\|_{L^\infty(\Omega_{i,r_n})}\leq C\left( \|u_{i,n}\|_{L^\infty(\Omega_{i,r_n})}+\|\lambda_{i,n}u_{i,n}\|_{L^\infty(\Omega_{i,r_n})}\right) .
\]
Since the right hand side is bounded by Lemma \ref{eq:uniformbounds_H^1_L^infty}, we obtain a contradiction.
\end{proof}

Now that we have established the behavior of the sequence $\{r_n\}$, we can introduce the quantities that will guide us in the proof of our main result. 

\begin{lemma}\label{lem: Rn o rn}
Let $C>0$ be the dimensional constant of Lemma \ref{lem max bound}. There exists a sequence $\{x_n\} \subset \{u_{1,n}>0\}$ such that
 \begin{equation}\label{eq:contradiction_assumption_xn}
	C M_n \leq |\nabla u_{1,n}(x_n)| \leq M_n
\end{equation}
and, moreover, 
\[
R_n := \dist(x_n,\pa \{u_{1,n}>0\}) =o(r_n)
\]
as $n \to \infty$.
\end{lemma}

\begin{proof}
We can directly apply Lemma \ref{lem max bound} to each function $u_{1,n}$ (for $n$ fixed) to obtain the desired result.
\end{proof}

Now, let $y_n \in \partial \{u_{1,n} > 0\}$ be a projection of $x_n$ onto $\partial\{u_{1,n}>0\}$, so that $R_n=|x_n-y_n|$. 
We shall analyze the behavior of the sequence $\{x_n\}$ and of $\{y_n\}$. As a first step, we show that the sequence $\{x_n\}$ is very close to the free-boundary $\partial \Omega_{1,n} \cap \Omega$ and not to the fixed  boundary of $\Omega$. This is the content of the next result.

\begin{lemma}\label{lem ratio}
We have that $\dist (x_n,\partial \Omega)/r_n \to +\infty$. In particular $y_n \in \partial \Omega_{1,n} \setminus \partial \Omega$ and, moreover,
\[
	M_n^2 \leq C \frac{1}{r_n^N} \int_{B_{r_n}(y_n)} |\nabla u_{1,n}|^2 
\]
for a constant $C = C(N, \bar \lambda)>0$ and $n$ sufficiently large.
\end{lemma}

\begin{proof} 
We prove this result by virtue of a contradiction argument. Assume that there exists a constant $\kappa > 0$ and a subsequence (which we shall not relabel) such that
\[
	\dist (x_n,\partial \Omega) \leq \kappa r_n.
\]	
We assume that $n$ is sufficiently large in such a way that $\tilde r(N, \lambda_{1,n}) \geq 1$.

\noindent \textbf{Case 1)} $y_n \in \partial \Omega$. In this case, by joining Corollaries \ref{cor mean} and \ref{cor acf one eig}, and recalling that $\Omega$ has the exterior sphere condition with radius at least $1$, we have that
\[
\begin{split}
	C M_n &\le |\nabla u_{1,n}(x_n)|^2 \leq C \frac{1}{R_n^N} \int_{B_{R_n}(x_n)} |\nabla u_{1,n}|^2 \leq 2^NC  \frac{1}{(2R_n)^N} \int_{B_{2R_n}(y_n)} |\nabla u_{1,n}|^2\\
	&\leq C \int_{B_{1}(y_n)} |\nabla u_{1,n}|^2 \leq C \|u_{1,n}\|^2_{H^1},
\end{split}
\]
and we find a contradiction with Lemma \ref{eq:uniformbounds_H^1_L^infty}, since $M_n \to +\infty$.

\noindent \textbf{Case 2)} $y_n \not\in \partial \Omega$. In this second case we need an additional step. Let $\rho_n = \dist (y_n,\partial \Omega)$ and $z_n \in \partial \Omega$ such that $|z_n-y_n| = \rho_n$. It is plain that $\rho_n \leq (1+\kappa)r_n$. At first, by using again Corollary \ref{cor mean}, and Corollary \ref{cor acf one eig} on balls centered in $y_n$, where $\{u_{1,n}>0\}$ has an exterior sphere of radius $r_n > 2R_n$, we obtain
\begin{equation}\label{2005nic}
\begin{split}
C M_n^2  &\le |\nabla u_{1,n}(x_n)|^2 \leq C \frac{1}{R_n^N} \int_{B_{R_n}(x_n)} |\nabla u_{1,n}|^2 \\
&\leq 2^NC  \frac{1}{(2R_n)^N} \int_{B_{2R_n}(y_n)} |\nabla u_{1,n}|^2 
\leq \frac{C}{r_n^N} \int_{B_{r_n}(y_n)} |\nabla u_{1,n}|^2.
\end{split}
\end{equation}
At this point, since $B_{r_n}(y_n) \subset B_{r_n+\rho_n}(z_n)$, $\rho_n \leq (1+\kappa)r_n$, and $\{u_{1,n}>0\}$ has, at $z_n$, an exterior ball or radius $1$, Corollary \ref{cor acf one eig} again yields
\[
\begin{split}
	M_n^2 & \le  \frac{1}{r_n^N} \int_{B_{r_n}(y_n)} |\nabla u_{1,n}|^2  \leq  C\frac{(r_n+\rho_n)^N}{r_n^N} \frac{1}{(r_n+\rho_n)^N} \int_{B_{r_n + \rho_n}(z_n)} |\nabla u_{1,n}|^2  \\
&\leq C(2+\kappa)^N \frac{1}{(r_n+\rho_n)^N} \int_{B_{r_n+\rho_n}(z_n)} |\nabla u_{1,n}|^2 \leq C \int_{B_{1}(z_n)} |\nabla u_{1,n}|^2 \leq  C \|u_{1,n}\|^2_{H^1},
\end{split}
\]
and we find again a contradiction with Lemma \ref{eq:uniformbounds_H^1_L^infty}.

This completes the proof of the first part of the statement. To obtain the desired estimate, one can now proceed as in \eqref{2005nic}.
\end{proof}

To proceed further, we recall the Caffarelli-Jerison-Kenig formula, a fundamental result for free-boundary problems \cite{CJK}. Let $u,v \in H^1(\R^N)$ be two continuous and non-negative functions such that $u(x)v(x)= 0$ for any $x \in \R^N$ 
and $\|u\|_{L^2} = \|v\|_{L^2} = 1$. Assume moreover that there exists a constant $M > 0$ such that
\[
-\Delta u \leq M, \quad -\Delta v \leq M \qquad \text{in $\R^N$}
\]
in the sense of measures. Then there exists $C = C(N,M)$ such that
\[
\frac{1}{r^2} \int_{B_r(x)} \frac{|\nabla u|^2 }{|x-y|^{N-2}} \cdot \frac{1}{r^2} \int_{B_r(x)} \frac{|\nabla v|^2}{|x-y|^{N-2}} \leq C	
\]
for any $x \in \R^N$ and $r \in (0,1)$. We can directly apply the Caffarelli-Jerison-Kenig formula to our setting, since $\{\mathbf{u}_n\}$ is uniformly bounded in $L^\infty(\Omega)$ and the eigenvalues $\{\lambda_{i,n}\}$ are uniformly bounded as well, see Lemma \ref{eq:uniformbounds_H^1_L^infty}. Thus, there exists a constant $C = C(N,\bar \lambda) > 0$ such that
\begin{equation}\label{eqn bound acf}
	\frac{1}{r^N} \int_{B_r(y_n)} |\nabla u_{1,n}|^2  \cdot \frac{1}{r^N} \int_{B_r(y_n)}  |\nabla u_{j,n}|^2 \le \frac{1}{r^2} \int_{B_r(y_n)} \frac{|\nabla u_{1,n}|^2 }{|x-y_n|^{N-2}} \cdot \frac{1}{r^2} \int_{B_r(y_n)} \frac{|\nabla u_{j,n}|^2}{|x-y_n|^{N-2}} \leq C
\end{equation}
for any $0 < r < 1$ and any $j \neq 1$. 

Now we introduce the following rescaled functions
\[
	\mathbf{v}_n(x) := \frac{\mathbf{u}_n(y_n + r_n x)}{r_n M_n}, \quad x \in \Omega_n:= \frac{\Omega - y_n}{r_n},
\]
extended as $0$ to $\R^N\setminus \Omega$.

Clearly, we have $\mathbf{v}_n \in H_0^1(\Omega_n)$, $\|\nabla \mathbf{v}_n \|_{L^\infty(\R^N)} \leq 1$, and $\mathbf{v}_n(0) = 0$, for every $n$. Each set $\{v_{i,n}>0\}$ enjoys the exterior sphere condition of the same radius 1. By Lemmas \ref{lem: Rn o rn} and \ref{lem ratio}, the sets $\Omega_n$ exhaust $\R^N$ as $n \to \infty$. Moreover,
\begin{equation}\label{eqn bu scaling int}
\begin{split}
	&\int_{\R^N} v_{i,n}^2 = \frac{1}{r_n^{N+2} M_n^2} \int_{\R^N} u_{i,n}^2 =  \frac{1}{r_n^{N+2} M_n^2},\\ 
	&\int_{\R^N} |\nabla v_{i,n}|^2 = \frac{1}{r_n^{N} M_n^2} \int_{\R^N} |\nabla u_{i,n}|^2 = \frac{1}{r_n^{N} M_n^2} \lambda_{i,n},
\end{split}
\end{equation}
and $\mathbf{v}_n$ is a minimizer for the following scaled version of problem \eqref{eq:weak_characterization}:
\begin{equation}\label{min v}
\inf\left\{ J(\mathbf{v}) : \ v_i \in H^1_0(\Omega_n)\setminus\{0\}\ \forall i,\ \dist(\text{supp}\, v_i,\text{supp}\, v_j)\geq 1,\ \forall i\neq j\right\},
\end{equation}
where
\[
J(\mathbf{v}) = \sum_{i=1}^k \frac{\int_{\Omega_n} |\nabla v_{i}|^2 }{\int_{\Omega_n} v_{i}^2} = \sum_{i=1}^k \frac{\int_{\R^N} |\nabla v_{i}|^2 }{\int_{\R^N} v_{i}^2}.
\]
The asymptotic properties of $\{\mathbf{v}_n\}$ are collected in the following statement.

\begin{lemma}\label{lem: asy vn}
There exists a globally Lipschitz function $\mathbf{v}=(v_1,\ldots, v_k)$ defined in $\R^N$, with Lipschitz constant $1$, such that:
\begin{itemize}
\item[(i)] $\mathbf{v}_n \to \mathbf{v}$ in $C^{0,\alpha}_{\loc}(\R^N)$, for every $\alpha \in (0,1)$, and strongly in $H^1_{\loc}(\R^N)$;
\item[(ii)] the first component $v_1$ is not identically $0$ in $B_1$, and moreover, for any $R\geq 1$ there exists a constant $C=C(R)$ such that
\begin{equation}\label{eqn h1 estimate bu 1}
	\int_{B_R} |\nabla v_{1,n}|^2 \geq C \frac{r_n^N M_n^2}{\lambda_{1,n}} \int_{\R^N} |\nabla v_{1,n}|^2.
\end{equation}
\item[(iii)] the other components $v_j$, $j \neq1$, vanish identically in $\R^N$, and moreover, for any $R>1$ there exists $C= C(R)$ such that
\begin{equation}\label{eqn h1 estimate bu}
	\int_{B_R} |\nabla v_{j,n}|^2 + v_{j,n}^2 \leq C \frac{r_n^N}{M_n^2 \lambda_{j,n}} \int_{\R^N} |\nabla v_{j,n}|^2.
\end{equation}
\end{itemize}
\end{lemma}

\begin{proof}
The $C^{0,\alpha}_{\loc}$ convergence $\mathbf{v}_n \to \mathbf{v}$ to some $\mathbf{v} \in \mathrm{Lip}_{\loc}(\R^N)$, with $\|\nabla v\|_{L^\infty(\R^N)} \le 1$, follows directly from the uniform gradient bound, and the fact that $\mathbf{v}_n(0) = 0$, via the Ascoli-Arzel\`a theorem. To show that the convergence is also strong in $H^1_{\loc}(\R^N)$, it is not difficult to adapt the argument in \cite[Lemma 3.11]{TT}: in fact, since $-\Delta v_{i,n}\leq r_n^2 \lambda_{i,n} v_{i,n}$ and $-\Delta v_{i}\leq 0$ in $\R^N$, there exists (local) nonnegative Radon measures $\mu_{i,n},\mu_i\in \mathcal{M}_{loc}(\R^N)$ such that
\[
-\Delta v_{i,n}=r_n^2\lambda_{i,n} v_{i,n}+ \mu_{i,n},\qquad -\Delta v_i=\mu_i,
\]
and since $v_{i,n}\rightharpoonup v_i$ weakly in $H^1_{loc}(\R^N)$, then $\mu_{i,n}\rightharpoonup \mu_i$ in the sense of measures $\mathcal{M}_{loc}(\R^N)$. Then, the argument follows by testing $-\Delta (v_{i,n}-v_i)=r_n^2\lambda_{i,n} v_{i,n} + \mu_{i,n}-\mu_i$ with $(v_{i,n}-v_i)\varphi$, for $\varphi\in C^\infty_c(\R^N)$. This proves (i). 
Concerning (ii), we just need to recall that
\[
	M_n^2 \leq C \frac{1}{r_n^N} \int_{B_{r_n}(y_n)} |\nabla u_{1,n}|^2
\]
by Lemma \ref{lem ratio}. This gives, by rescaling and passing to the limit in $n$, that
\[
	 \int_{B_{1}} |\nabla v_{1,n}|^2 \geq \frac{1}{C} \quad \implies \quad \int_{B_{1}} |\nabla v_{1}|^2 \geq \frac{1}{C}
\]
and hence $v_1 \not \equiv 0$ in $B_1$. Furthermore, combining this estimate and \eqref{eqn bu scaling int}, we obtain \eqref{eqn h1 estimate bu 1}. It remains to prove the validity of point (iii). By scaling \eqref{eqn bound acf}, we have that for any $R > 1$ and $n$ sufficiently large,
\[
	\int_{B_R} |\nabla v_{j,n}|^2 \le C \int_{B_1} |\nabla v_{1,n}|^2 \int_{B_R} |\nabla v_{j,n}|^2 
	\le C \int_{B_R} |\nabla v_{1,n}|^2 \int_{B_R} |\nabla v_{j,n}|^2 \le C R^{2N} M_n^{-4}
\]
that is, 
\begin{equation}\label{eqn h1 estimate bu j}
	\int_{B_R} |\nabla v_{j,n}|^2  \le C R^{2N} M_n^{-4}\to 0.
\end{equation}
Thus, for any $R>1$, the sequence $\{v_{j,n}\}$ converges in $H^1(B_R)$ and in $C^{0,\alpha}(B_R)$ (for any $\alpha \in (0,1)$) to a constant $c$. But since $v_{j,n}^2(0)=0$, necessarily the limit $c=0$ and, since $R>1$ was arbitrarily fixed, $v_j \equiv 0$ in $\R^N$ for every $j=2,\dots,k$. Moreover, by Lemma \ref{lem poinc} we have
\[
	\int_{B_R} (|\nabla v_{j,n}|^2 + v_{j,n}^2) \leq \left(1+ C_P\right) \int_{B_R} |\nabla v_{j,n}|^2
\]
for a constant $C_P=C_P(N,R)$ that depends only on the dimension $N$ and on the fixed radius $R$. On the other hand, recalling \eqref{eqn bu scaling int} and \eqref{eqn h1 estimate bu j}, we find
\[
	\frac{\displaystyle \int_{B_R} |\nabla v_{j,n}|^2 }{ \displaystyle \int_{\R^N} |\nabla v_{j,n}|^2} \leq \frac{C R^{2N} M_n^{-4} }{\frac{1}{r_n^{N} M_n^2} \lambda_{j,n} } \implies \int_{B_R} |\nabla v_{j,n}|^2  \leq C \frac{r_n^N}{M_n^2 \lambda_{j,n}} \int_{\R^N} |\nabla v_{j,n}|^2
\]
Putting these last two inequalities together, estimate \eqref{eqn h1 estimate bu} follows.
\end{proof}
Point (iii) of the previous lemma establishes that the energy of each $v_{j,n}$, with $j\geq 2$, ``escapes to infinity'': thus, whenever we remove mass from a fixed ball and distribute it on the remainder of the domain, the $H^1$-norm should not increase in a significant way. We can be more precise. Let $\rho > 0$ be a fixed large positive number and let $\eta$ be the defined by
	\[
	\eta(x) := \begin{cases}
	1 & \text{if $|x|> 2 + \rho$}\\
	|x| - (1+\rho) & \text{if $1+\rho \leq |x| \leq 2 + \rho$}\\
	0 & \text{if $|x|< 1 + \rho$}.
	\end{cases}
	\]
We point out that $0\leq \eta \leq 1$, $|\nabla \eta|\leq 1$. Let also $\bar v_{j,n} := \eta v_{j,n}$, for $j\geq 2$. We have that $\bar v_{j,n} = v_{j,n}$ in $\R^N \setminus B_{2+\rho}$, while $\bar v_{j,n} \le v_{j,n}$ in $B_{2+\rho}$, and actually the support of $v_{j,n}$ is ``cut" by the multiplication with $\eta$. In the next lemma we estimate the energy gap between $\bar v_{j,n}$ and $v_{j,n}$.

\begin{lemma}\label{lem: en vjn}
Let $\delta_n  :=  r_n^N / M_n^2$, which tends to $0$ as $n \to \infty$. There exists $C>0$ such that
\[
\frac{\displaystyle\int_{\R^N} |\nabla\bar v_{j,n}|^2 }{ \displaystyle	\int_{\R^N} \bar v_{j,n}^2 } \leq \left(1+\frac{C}{\lambda_{j,n}}\delta_n\right) \frac{\displaystyle\int_{\R^N} |\nabla v_{j,n}|^2 }{ \displaystyle	\int_{\R^N}  v_{j,n}^2 } 
\]
for every $n$ sufficiently large.
\end{lemma}
\begin{proof}
Recalling that $\|v_{j,n}\|_{H^1(B_R)} \to 0$ for all $R> 0$, we find that 
\[
	\int_{\R^N} v_{j,n}^2 = \int_{\R^N} \eta^2 v_{j,n}^2 + \int_{\R^N} \left(1- \eta^2\right) v_{j,n}^2 \leq \int_{\R^N} \bar v_{j,n}^2 + \int_{B_{2+\rho}} v_{j,n}^2 = \int_{\R^N} \bar v_{j,n}^2 +\frac{C}{\lambda_{j,n}} \delta_{n} \int_{\R^N} |\nabla v_{j,n}|^2,
\]
where in the last step we used estimate \eqref{eqn h1 estimate bu} (notice that $C$ depends on $\rho$, which is fixed).  Similarly, we find
\[
\begin{split}
	\int_{\R^N} |\nabla\bar v_{j,n}|^2 & = \int_{\R^N} \left( |\nabla \eta|^2 v_{j,n}^2 + 2\eta  v_{j,n} \nabla \eta \cdot \nabla v_{j,n} + \eta^2 | \nabla v_{j,n}|^2\right)\\
	& =  \int_{\R^N} | \nabla v_{j,n}|^2  + \int_{B_{2+\rho}} \left[(\eta^2 -1) | \nabla v_{j,n}|^2 + |\nabla \eta|^2 v_{j,n}^2 + 2\eta  v_{j,n} \nabla \eta \cdot \nabla v_{j,n} \right]\\
	& \leq  \int_{\R^N} | \nabla v_{j,n}|^2  + \int_{B_{2+\rho}}  \left[ (\eta^2 -1) | \nabla v_{j,n}|^2 +2|\nabla \eta|^2 v_{j,n}^2 +    \eta^2 |\nabla v_{j,n} |^2\right]
\\	& \leq  \int_{\R^N} | \nabla v_{j,n}|^2  + 2 \int_{B_{2+\rho}} \left(  v_{j,n}^2 + | \nabla v_{j,n}|^2 \right)\leq \left(1+\frac{C}{\lambda_{j,n}}\delta_n\right)  \int_{\R^N} | \nabla v_{j,n}|^2.
\end{split}
\]
As a result, combining this with \eqref{eqn bu scaling int} and recalling from Lemma \ref{eq:uniformbounds_H^1_L^infty} that the eigenvalues $\lambda_{j,n}$ are bounded from above and away from $0$, we obtain that
\[
	\frac{\displaystyle\int_{\R^N} |\nabla\bar v_{j,n}|^2 }{ \displaystyle	\int_{\R^N} \bar v_{j,n}^2 } \leq \frac{\displaystyle \left(1+\frac{C}{\lambda_{j,n}}\delta_n\right)  \int_{\R^N} |\nabla v_{j,n}|^2 }{ \displaystyle	\int_{\R^N} v_{j,n}^2 - \frac{C}{\lambda_{j,n}}\delta_{n} \int_{\R^N} |\nabla v_{j,n}|^2 } = \frac{1+\frac{C}{\lambda_{j,n}}\delta_n}{1 - C \delta_n r_n^2 }  \frac{\displaystyle\int_{\R^N} |\nabla v_{j,n}|^2 }{ \displaystyle	\int_{\R^N} v_{j,n}^2 } \leq \left(1+\frac{C}{\lambda_{j,n}}\delta_n\right) \frac{\displaystyle\int_{\R^N} |\nabla v_{j,n}|^2 }{ \displaystyle	\int_{\R^N}  v_{j,n}^2 } 
\]
for $n$ sufficiently large, which is the desired result.
\end{proof}

Now the idea is to construct a competitor for $\mathbf{v}_n$ with lower energy $J$. This will be in contradiction with the fact that $\mathbf{v}_n$ is a minimizer for \eqref{min v}, and will complete the proof. The $j$-th component of the competitor will be $\bar v_{j,n}$, for $j \ge 2$. We need to conveniently define the first component $\bar v_{1,n}$, and the idea is to enlarge the support of $v_{1,n}$ (taking advantage of the fact that the support of $v_{j,n}$ was previously cut, for $j \ge 2$), in order to substantially lower the Rayleigh quotient of $v_{1,n}$. We present the details in what follows.

We have already established that, in any ball $B_R$ with $R>1$, the function $v_1$ is not identically 0. Moreover, $0 \in \partial \{v_{1,n}>0\}$ for every $n$, and $\{v_{1,n}>0\}$ satisfies the exterior sphere condition of radius $1$ at $0$, and, in the exterior sphere, we have $v_{1,n} \equiv 0$. Up to a rotation, it is not restrictive to suppose that $B_1(e_1)$ is such exterior sphere. We consider a new sequence of functions $\bar v_{1,n} \in H^1_0(\Omega_n)$ defined piece-wise as follows:
\begin{itemize}
	\item for $|x| \geq \rho$, we let $\bar v_{1,n}(x) =  v_{1,n}(x)$;
	\item for $|x| < \rho$, we let $\bar v_{1,n}$ be such that
	\begin{equation}\label{var car bar v}
		\bar v_{1,n} = \arg\min \left\{ \int_{B_{\rho}} |\nabla v|^2 : \ v -  v_{1,n} \in H_0^1(B_\rho), \ \int_{B_\rho} v^2  = \int_{B_\rho} v_{1,n}^2 \right\}
		\end{equation}
		\end{itemize}		
Since $v_{1,n} \equiv 0$ in $B_\rho \cap B_1(e_1)$, the support of $\bar v_{1,n}$ is strictly larger than the one of $v_{1,n}$, and it is at distance at least 1 from the support of $\bar v_{j,n}$, for any $j \geq 2$ (by definition of $\eta$). Moreover, we have that
\[
	\int_{\R^N} \bar v_{1,n}^2 = \int_{\R^N} v_{1,n}^2 \qquad \text{while} \qquad \int_{\R^N} |\nabla \bar v_{1,n}|^2 < \int_{\R^N} |\nabla v_{1,n}|^2.
\]
Concerning the last inequality, we have to be more precise.

\begin{lemma}\label{lem: en v_1}
There exists $\eps \in (0,1)$ such that
\begin{equation}\label{scarto energia}
	\int_{B_\rho}|\nabla \bar v_{1,n}|^2 \leq (1-\eps) \int_{B_\rho}|\nabla v_{1,n}|^2 \qquad \forall n.
\end{equation}
\end{lemma}
\begin{proof}
The proof is quite long and, for the reader's convenience, we divide it into some intermediate steps. Assume by contradiction that, up to striking out a subsequence,
\begin{equation}\label{0405nic}
	 (1-\eps_n) \int_{B_\rho}|\nabla v_{1,n}|^2 \leq \int_{B_\rho}|\nabla \bar v_{1,n}|^2 \leq \int_{B_\rho}|\nabla v_{1,n}|^2,
\end{equation}
with $\eps_n \to 0^+$. Since $\{\bar v_{1,n}\}$ is bounded in $H^1(B_\rho)$, $v_{1,n} = \bar v_{1,n}$ on $\pa B_\rho$, $\|v_{1,n}\|_{L^2(B_\rho)} = \|\bar v_{1,n}\|_{L^2(B_\rho)}$, and $v_{1,n} \to v_1 \not \equiv 0$ strongly in $H^1(B_\rho)$ (see Lemma \ref{lem: asy vn}), we have that up to a subsequence $\bar v_{1,n} \rightharpoonup \bar v$ weakly in $H^1(B_\rho)$, strongly in $L^2(B_\rho)$ and in $L^2(\partial B_\rho)$ (by compactness of the trace operator $H^1(B_\rho) \to L^2(\pa B_\rho)$), with  $v_1 = \bar v$ on $\pa B_\rho$, and $\int_{B_\rho} \bar v_1^2 =: c>0$. Moreover, by minimality and the strong maximum principle
\[
-\Delta \bar v_{1,n}= \bar \lambda_{n} \bar v_{1,n}, \quad \bar v_{1,n}>0 \qquad \text{in $B_\rho$},
\]
for some $\bar \lambda_n \in \R$.\\
\emph{Step 1)} The sequence $\{\bar \lambda_n\}$ is bounded. To prove this claim, we first show that there exists $r \in (0,\rho)$ and a subsequence $n_k \to +\infty$ such that, 
\begin{equation}\label{stima dal basso su palletta}
\int_{B_{r}} \bar v_{1,n_k}^2 \ge \frac{c}{2} \qquad  \text{for every $k$}.
\end{equation} 
Indeed, if by contradiction this were not true, we would have that 
\[
\int_{B_r} \bar v_{1,n}^2 < \frac{c}{2} \quad \text{for every $r \in (0,\rho)$, for every $n$ large.}
\]
But, in this case, if $r_m \to \rho^-$, with a diagonal selection we could find an increasing sequence $n_m \to \infty$ such that
\[
\int_{B_{r_m}} \bar v_{1,n_m}^2 <  \frac{c}{2} \quad \text{for every $m$;}
\]
then, by strong convergence,
\[
\frac{c}{2} > \int_{B_{r_m}} \bar v_{1,n_m}^2 = \int_{B_\rho} \bar v_{1,n_m}^2 \chi_{B_{r_m}} \to \int_{B_\rho} \bar v_{1}^2  = c >0
\]
a contradiction. Now, denoting for the sake of simplicity by $\{\bar v_{1,n}\}$ the sequence in \eqref{stima dal basso su palletta}, let us take a non-negative $\varphi \in C^\infty_c(B_\rho)$, with $\varphi \equiv 1$ on $B_r$, and let us test the equation of $\bar v_{1,n}$ with $\bar v_{1,n} \varphi^2$: we obtain
\[
\begin{split}
 |\bar \lambda_n| \int_{B_\rho} \bar v_{1,n}^2 \varphi^2 = \left|\int_{B_\rho} |\nabla \bar v_{1,n}|^2 \varphi^2 + 2 \bar v_{1,n} \varphi \nabla \bar v_{1,n} \cdot \nabla \varphi\right| \le C \|\bar v_{1,n}\|_{H^1(B_\rho)}^2. 
\end{split}
\]
Since the coefficient of $ |\bar \lambda_n|$ is bounded from below, by \eqref{stima dal basso su palletta}, and $\{\bar v_{1,n}\}$ is bounded in $H^1(B_\rho)$, this implies that $\{\lambda_n\}$ is bounded.\\
\emph{Step 2)} $\bar v_{1,n} \to \bar v$ strongly in $H^1(B_\rho)$. Let $\bar v_{1,n} = w_n + v_{1,n}$, by linearity we have that the sequence $\{w_n\} \subset H^1_0(B_\rho)$ converges weakly in $H^1_0(B_\rho)$ and strongly in $L^2(B_\rho)$ to $\bar v - v_1$. Moreover, for any $n \in \N$, $w_n$ solves
\[
	\int_{B_\rho} \nabla w_n \cdot \nabla \varphi - \bar \lambda_{n} w_n \varphi + \int_{B_\rho} \nabla v_{1,n} \cdot \nabla \varphi - \bar \lambda_{n} v_{1,n} \varphi = 0 \qquad \forall \varphi \in H^1_0(\Omega).
\]
That is, for any $n, m \in \N$, we have
\[
	\int_{B_\rho} \nabla (w_n -w_m) \cdot \nabla \varphi + \int_{B_\rho} \nabla (v_{1,n}-v_{1,m}) \cdot \nabla \varphi - \left(\bar \lambda_{n} v_{1,n} - \lambda_{m} v_{1,m}\right) \varphi +  \left(\bar \lambda_{n} w_n - \bar \lambda_{m} w_m\right) \varphi = 0,
\]
for every $\varphi \in H^1_0(\Omega)$. Taking $\varphi = w_n-w_m \in H^1_0(B_\rho)$ yields
\begin{multline*}
\int_{B_\rho} |\nabla (w_n -w_m)|^2 = -\int_{B_\rho} \nabla (v_{1,n}-v_{1,m}) \cdot \nabla (w_n -w_m) \\
-\int_{B_\rho} \left(\bar \lambda_{n} w_n - \bar \lambda_{m} w_m -\bar \lambda_{n} v_{1,n} + \lambda_{m} v_{1,m}\right) (w_n -w_m).
\end{multline*}
Now we recall that $v_{1,n} \to v$ in $H^1(B_\rho)$, $\bar v_{1,n} \rightharpoonup \bar v$  in $H^1(B_\rho)$, and that the sequence $\{\bar \lambda_n\}$ is bounded, as proved in Step 1. Thus, we deduce that the right hand side of the previous equation converges to 0 as $m,n \to +\infty$. That is, $\{w_n\}$ is a Cauchy sequence in $H^1(B_\rho)$, and we conclude by linearity that $\bar v_{1,n} \to \bar v$ strongly in $H^1(B_\rho)$, as claimed.\\
\emph{Step 3)} We are ready to prove that \eqref{0405nic} gives a contradiction, which entails the validity of estimate \eqref{scarto energia}. Recall the variational characterization of $\bar v_{1,n}$, given in \eqref{var car bar v}. Collecting what we proved so far, we have that $\bar v_{1,n} \to \bar v$ strongly in $H^1(B_\rho)$, where $\bar v$ satisfies, for some $\bar \lambda \in \R$,
\[
\begin{cases} -\Delta \bar v = \bar \lambda \bar v, \quad \bar v >0 & \text{in $B_\rho$}\\ \bar v= v_1 & \text{on $\pa B_\rho$;} \end{cases}
\]
moreover $\|\bar v\|_{L^2(B_\rho)} = \|v_1\|_{L^2(B_\rho)}$. We claim that $\bar v$ minimizes
\[
\inf\left\{ \int_{B_\rho}  |\nabla w|^2 : \ w- v_1 \in H_0^1(B_\rho), \  \int_{B_\rho} w^2  = \int_{B_\rho} v_{1}^2 \right\}.
\]
The desired contradiction follows easily from this claim: by \eqref{0405nic} and the strong convergence $v_{1,n} \to v_1$, we would have that also $v_1$ is a nonnegative minimizer for the same problem. But any nonnegative minimizer solves
\[
 -\Delta v = \lambda v, \quad  v >0 \qquad \text{in $B_\rho$}
\]
for some $\lambda>0$, which is in contradiction with the fact that $v_1 \equiv 0$ in $B_\rho \cap B_1(e_1)$. To prove that $\bar v$ is a minimizer, we argue again by contradiction, and suppose that there exists $w \in H^1(B_\rho)$ such that 
\[
w-v_1 \in H_0^1(B_\rho), \quad \int_{B_\rho} w^2 = \int_{B_\rho} v_1^2, \quad  
\int_{B_\rho} |\nabla w|^2 < \int_{B_\rho} |\nabla \bar v|^2. 
\]
In this case, take 
\[
z_n = w + (\bar v_{1,n}-\bar v) + t_n \varphi \quad \text{with }\varphi \in C^\infty_c(B_\rho): \ \int_{B_\rho} w \varphi >0.
\]
and $t_n \to 0$ to be chosen later. It is plain that $z_n-\bar v_{1,n} \in H_0^1(B_\rho)$, with $z_n \to w$ strongly in $H^1(B_\rho)$. Thus, by strong convergence,
\[
\int_{B_\rho} |\nabla w|^2 < \int_{B_\rho} |\nabla \bar v|^2 \quad \implies \quad \int_{B_\rho} |\nabla z_n|^2 < \int_{B_\rho} |\nabla \bar v_{1,n}|^2
\]
for every $n$ large enough. Now we show that we can choose $t_n$ in such a way that $\int_{B_\rho} z_n^2 = \int_{B_\rho} \bar v_{1,n}^2$. In fact, to impose such a condition $\int_{B_\rho} z_n^2 = \int_{B_\rho} \bar v_{1,n}^2$ amounts to require that
\[
 \int_{B_\rho} \bar v^2 +\int_{B_\rho} (\bar v_{1,n}- \bar v)^2 + t_n^2 \int_{B_\rho} \varphi^2 + 2 t_n\left( \int_{B_{\rho}} w \varphi + (\bar v_{1,n}- \bar v) \varphi\right)  + 2 \int_{B_\rho} (\bar v_{1,n}- \bar v) w = \int_{B_\rho} \bar v_{1,n}^2.
\]
This is an equation of type
\[
t_n^2 + a_n t_n + b_n = 0 \quad \text{with $a_n \to a>0$ and $b_n \to 0$}, 
\]
where we used the fact that $\int_{B_\rho} w \varphi>0$ by assumption, and the convergence of $\bar v_{1,n}$ to $\bar v$. Such an equation clearly admits a solution $t_n \to 0$. To sum up, we showed that $z_n$ is an admissible competitor for $\bar v_{1,n}$ with a lower energy, in contradiction with the minimality of $\bar v_{1,n}$. The contradiction shows that $w$ as above cannot exist, that is, $\bar v_1$ is a minimizer. As observed, this completes the proof of the lemma.
\end{proof}

\begin{proof}[Conclusion of the proof of Theorem \ref{thm:unif Lip}]
As consequence of the Lemma \ref{lem: en v_1}, we can give a quantitative estimate for the energy gap between $v_{1,n}$ and $\bar v_{1,n}$. Indeed, exploiting also \eqref{eqn h1 estimate bu 1}, we have that
\[
	\begin{split}
	\int_{\R^N}|\nabla \bar v_{1,n}|^2 &= \int_{B_\rho}|\nabla \bar v_{1,n}|^2 + \int_{\R^N \setminus B_\rho}|\nabla \bar v_{1,n}|^2 \leq (1-\eps)\int_{B_\rho}|\nabla v_{1,n}|^2 + \int_{\R^N \setminus B_\rho}|\nabla v_{1,n}|^2 \\
	&\leq \int_{\R^N}|\nabla v_{1,n}|^2 -\eps \int_{B_\rho}|\nabla v_{1,n}|^2 \\
	&\leq \int_{\R^N}|\nabla v_{1,n}|^2 -\eps C \frac{r_n^N M_n^2}{\lambda_{1,n}} \int_{\R^N} |\nabla v_{1,n}|^2 = \left(1 - \eps C \frac{ M_n^4}{\lambda_{1,n}} \delta_n \right) \int_{\R^N} |\nabla v_{1,n}|^2
	\end{split}
\]
where $\delta_n = r_n^N / M_n^2$, as in Lemma \ref{lem: en vjn}, and $C>0$ is a positive constant independent of $n$. Combining this estimate with Lemma \ref{lem: en vjn}, we can finally prove that the competitor $\bar{\mathbf{v}}_n=(\bar v_{1,n}, \bar v_{2,n},\dots, \bar v_{k,n})$ has lower energy than $\mathbf{v}_n$: indeed
\[
	\begin{split}
		J(\bar{\mathbf{v}}_n) &= \sum_{i=1}^k \frac{\displaystyle\int_{\R^N} |\nabla\bar v_{i,n}|^2 }{ \displaystyle	\int_{\R^N} \bar v_{i,n}^2 } = 	\frac{\displaystyle\int_{\R^N} |\nabla\bar v_{1,n}|^2 }{ \displaystyle	\int_{\R^N} \bar v_{1,n}^2 } + \sum_{j=2}^k 	\frac{\displaystyle\int_{\R^N} |\nabla\bar v_{j,n}|^2 }{ \displaystyle	\int_{\R^N} \bar v_{j,n}^2 }\\
		&\leq \left(1 - \eps C \frac{ M_n^4}{\lambda_{1,n}} \delta_n \right) \frac{\displaystyle\int_{\R^N} |\nabla v_{1,n}|^2 }{ \displaystyle	\int_{\R^N} v_{1,n}^2 } + \sum_{j=2}^k \left(1+\frac{C}{\lambda_{j,n}}\delta_n\right) \frac{\displaystyle\int_{\R^N} |\nabla v_{j,n}|^2 }{ \displaystyle	\int_{\R^N} v_{j,n}^2 }\\
		&\leq J(\mathbf{v}_n)	 + \left( - \eps C \frac{ M_n^4}{\lambda_{1,n}} \delta_n r^2_n\lambda_{1,n} +  \sum_{j=2}^N \frac{C}{\lambda_{j,n}}\delta_n r^2_n \lambda_{j,n} \right)\\
		&\leq J(\mathbf{v}_n) + \delta_n r^2_n \left( C (k-1)  - \eps C  M_n^4  \right) < 	J(\mathbf{v}_n)	\end{split}
\]
for every $n$ large, where in the last step we have exploited the fact that $M_n \to +\infty$, as by assumption. This is a contradiction with the minimality of $\mathbf{v}_n$, and completes the proof of the Lipschitz bound in Theorem \ref{thm:unif Lip}.
\end{proof}

We now pass to the proof of Corollary \ref{coro:unif Lip}. We start with an estimate which implies the convergence of $c_r$ to $c_0$ as $r\to 0$.

\begin{lemma}\label{eq:estimate}
	There exists a constant $C > 0$ such that
	\[
		c_0 \leq c_r \leq c_0 + C r
	\]
for any $r>0$.
\end{lemma}
\begin{proof} The estimate $c_0\leq c_r$ is straightforward, so we prove the second one. We will use the solution of the problem $c_0$ (that is with distance constraint $r=0$) to construct a competitor for $c_r$ with $r > 0$.

Let $\mathbf{u}\in H^1_0(\Omega)$ be any minimizer of the problem $c_0$, and recall from Theorem A that $\mathbf{u} \in \mathrm{Lip}(\overline{\Omega})$. We denote $K = \max_{i= 1, \dots, N} \|\nabla u_i\|_{L^\infty(\Omega)}$. For any $i = 1, \dots, N$ we let $\Omega_i = \{u_i > 0\}$ so that $\Omega_i \cap \Omega_j = \emptyset$, and we have
\[
    \int_{\Omega} u_i^2 = 1, \quad \int_{\Omega} |\nabla u_i|^2 = \lambda_1(\Omega_i).
\]
We also recall from Theorem A that the free-boundary $\mathcal{N} := \{x \in \Omega : \mathbf{u}(x) = 0\} = \Omega \setminus (\cup_i \Omega_i)$ is an $(N-1)$-rectifiable set of finite $(N-1)$-Hausdorff measure. In particular, by the rectifiability of $\mathcal{N}$, we have that the Minkowski content of $\mathcal{N}$ coincides with its $N-1$-Hausdorff measure (\cite[Thm 3.2.39]{Federer}). More explicitly, if for a given $r>0$ we denote the $r$-tubular neighborhood of $\mathcal{N}$ as
\[
	\mathcal{N}_r = \{x \in \R^N : \dist(x, \mathcal{N}) < r\}
\]
then we have
\[
	\lim_{r \to 0^+} \frac{|\mathcal{N}_r|}{2r} = \mathcal{H}_{N-1}(\mathcal{N})
\]
where $|\cdot|$ is the Lebesgue measure in $\R^n$. 
 
Fix now a constant $C \geq 2$ such that for any $r>0$ sufficiently small we have
\[
	|\mathcal{N}_r| \leq C r \mathcal{H}_{N-1}(\mathcal{N}).
\]
We define a cutoff function $\eta \in Lip(\Omega)$ as follows
\[
	\eta(x) = \begin{cases}
		0 &\text{if $x\in \mathcal{N}_{r/2}$},\\
		\frac{\dist(x, \mathcal{N}) - r/2}{r/2} &\text{if $\dist(x, \mathcal{N}) \in [r/2,r]$},\\
		1 &\text{if $\Omega\setminus \mathcal{N}_r$}.
	\end{cases}
\]
Observe that $0 \leq \eta(x) \leq 1$ and we have, for a.e.\ $x \in \Omega$,
\[
|\nabla\eta(x)| = \begin{cases}
	0 &\text{if $\dist(x, \mathcal{N}) < r/2$ or $\dist(x, \mathcal{N}) > r$}\\
	\frac{2}{r} &\text{if $\dist(x, \mathcal{N}) \in [r/2,r]$}
\end{cases}
\]
We claim that, for $r>0$ small enough, the function $\mathbf{u} \eta \in H^1_0(\Omega)$ is an admissible competitor for the functional with distance constraint greater than or equal to $0$, and moreover that
\[
	J(\mathbf{u} \eta) \leq J(\mathbf{u}) + C r \mathcal{H}_{N-1}(\mathcal{N}).
\]
Indeed, by construction we immediately see that, for any $i \neq j$,
\[
	\dist(\supp (u_i \eta), \supp (u_j \eta)) \geq r.
\]
Thus $\mathbf{u} \eta$ is an admissible competitor for $c_r$. In order to estimate the energy
\[
	J(\mathbf{u} \eta) = \sum_{i=1}^k \frac{\displaystyle \int_{\Omega_i} |\nabla (u_i \eta)|^2}{\displaystyle \int_{\Omega_i} |u_i \eta|^2}
\]
we proceed separately for each component. We start with the denominator corresponding to the function $u_i \eta$, for which we can write
\[
	\int_{\Omega_i} |u_i \eta|^2 = \int_{\Omega_i} |u_i|^2 - \int_{\Omega_i \cap \mathcal{N}_r} |u_i|^2 (1 -|\eta|^2) = 1 - \int_{\Omega_i \cap \mathcal{N}_r} |u_i|^2 (1 -|\eta|^2) 
\]
Since $|\nabla u_i|\leq K$, we find that
\[
	|u_i(x)| \leq \dist(x, \mathcal{N}) K \qquad \forall x \in \Omega
\]
thus, recalling that $0\leq\eta(x)\leq 1$ for any $x\in\Omega$, we can carry on with the estimate as follows
\[
	\int_{\Omega_i \cap \mathcal{N}_r} |u_i|^2 (1 -|\eta|^2) \leq \int_{\Omega_i \cap \mathcal{N}_r } \dist(x, \mathcal{N})^2 K^2 \leq K^2 r^2 \left| \Omega_i \cap \mathcal{N}_r \right|.
\]
Finally we find that, for $r>0$ small
\[
	\left(\int_{\Omega_i} |u_i \eta|^2 \right)^{-1} \leq 1 + 2 K^2 r^2 \left| \Omega_i \cap \mathcal{N}_r \right|.
\]
Concerning the numerator of the Rayleigh quotient, we get
\[
\begin{split}
	\int_{\Omega_i} |\nabla (u_i \eta)|^2 &= \int_{\Omega_i} |\nabla u_i|^2 + \int_{\Omega_i \cap \mathcal{N}_r} (|\nabla(u_i\eta)|^2  - |\nabla u_i|^2)\\
	&= \lambda_1(\Omega_i) + \int_{\Omega_i \cap \mathcal{N}_r} (|\nabla u_i|^2 (|\eta|^2-1) + |u_i|^2|\nabla \eta|^2 + 2 u_i \eta \nabla u_i \cdot \nabla \eta) \\
	&\leq \lambda_1(\Omega_i) + \int_{\Omega_i \cap \mathcal{N}_r} (|u_i|^2|\nabla \eta|^2 + 2 u_i \eta \nabla u_i \cdot \nabla \eta).
\end{split}
\]
We estimate the two remaining terms separately. For the first one we have
\[
	\int_{\Omega_i \cap \mathcal{N}_r} |u_i|^2|\nabla \eta|^2 \leq \int_{\Omega_i \cap \{r/2 < \dist(x,\mathcal{N}) < r\} } \dist(x, \mathcal{N})^2 K^2 \frac{4}{r^2} \leq 4 K^2 \left| \Omega_i \cap \mathcal{N}_r \right|.
\]
For the second one, in a similar fashion, we obtain
\[
\begin{split}
	\int_{\Omega_i \cap \mathcal{N}_r} 2 u_i \eta \nabla u_i \cdot \nabla \eta &\leq \int_{\Omega_i \cap \mathcal{N}_r} 2 u_i |\nabla u_i| |\nabla \eta| \\
	&\leq  \int_{\Omega_i \cap \{r/2 < \dist(x,\mathcal{N}) < r\} } 2 \dist(x, \mathcal{N}) K^2 \frac{2}{r} \leq 4 K^2 \left| \Omega_i \cap \mathcal{N}_r\right|.
\end{split}
\]
As a result, recollecting the two estimates, we obtain
\[
	\int_{\Omega_i} |\nabla (u_i \eta)|^2 \leq \lambda_1(\Omega_i) +  8 K^2 \left| \Omega_i \cap \mathcal{N}_r\right|.
\]

We can control the Rayleigh quotient of $u_i\eta$ by combining the previous estimates. We get
\[
	\frac{\displaystyle \int_{\Omega_i} |\nabla (u_i \eta)|^2}{\displaystyle \int_{\Omega_i} |u_i \eta|^2} \leq \lambda_1(\Omega_i) +  C K^2 \left| \Omega_i \cap \mathcal{N}_r \right| + R(r)
\]
where $C$ is a constant independent of $r$ and $R(r)$ is a remainder term of higher order. Summing up in $i$ we find
\begin{multline*}
	J(\mathbf{u} \eta) = \sum_{i=1}^k \frac{\displaystyle \int_{\Omega_i} |\nabla (u_i \eta)|^2}{\displaystyle \int_{\Omega_i} |u_i \eta|^2} \leq \sum_{i=1}^k \lambda_1(\Omega_i) + C K^2 \sum_{i=1}^k \left| \Omega_i \cap \left\{ \dist(x,\mathcal{N}) < r\right\} \right| \\
	\leq J(\mathbf{u}) + CK^2 |\mathcal{N}_r| \leq  J(\mathbf{u}) + C 2r \mathcal{H}_{N-1}(\mathcal{N}).
\end{multline*}

To conclude, it suffices to remark that, since $\mathbf{u}\eta$ is an admissible competitor for $c_r$, 
\[
	c_r \leq J(\mathbf{u} \eta) \leq J(\mathbf{u}) + C 2r \mathcal{H}_{N-1}(\mathcal{N}) = c_0 + C 2r \mathcal{H}_{N-1}(\mathcal{N}). \qedhere
\]
\end{proof}

\begin{proof}[Proof of Corollary \ref{coro:unif Lip}]
Let $\mathbf{u}_r$ be a minimizer for problem $c_r$, with $r>0$ sufficiently small. From Theorem \ref{thm:unif Lip}, Lemma \ref{eq:estimate} and the fact that $\mathbf{u}_r|_{\partial \Omega}=0$, there exists $\mathbf{u}\in \mathrm{Lip}(\overline \Omega)$ such that, up to a subsequence,
\[
\mathbf{u}_r\to \mathbf{u}\quad \text{ weakly in $H^1(\Omega)$, strongly in $C^{0,\alpha}(\overline \Omega)$}.
\]
This implies that $\|u_{i,0}\|_{L^2(\Omega)} =1$ for every $i$, $\int_{\Omega} u_{i,0}^2 u_{j,0}^2 =0$ and $\dist(\supp u_{i,0}, \supp u_{j,0}) =0$ for every $i \neq j$; thus, $\mf{u}_0$ is an admissible competitor for $c_0$. Moreover, since $\|\nabla u_{i,0}\|_{L^2(\Omega)}\leq \liminf_{r\to 0} \|\nabla u_{i,r}\|_{L^2(\Omega)}$ for every $i$,
\[
c_0=\lim_{r\to 0} c_r = \lim_{r\to 0} \sum_{i=1}^k\int_{\Omega} |\nabla u_{i,r}|^2\geq \sum_{i=1}^k\int_{\Omega} |\nabla u_{i,0}|^2\geq c_0,
\]
which shows that $\mathbf{u}_0$ achieves $c_0$, and that $\mathbf{u}_r\to \mathbf{u}_0$ strongly in $H^1_0(\Omega)$.
\end{proof}

\begin{remark}
	The proof of Theorem \ref{thm:unif Lip 2} (the case of singularly perturbed harmonic maps with distance constraint) follows by similar arguments, with few differences (for instance the corresponding results in Section \ref{sec: pre} are much easier to prove). In particular, functions $\mathbf{u}_r$ are not zero on $\partial \Omega$ and we cannot achieve the first conclusion of Lemma \ref{lem ratio}, i.e., the sequence $\{x_n\}$ may accumulate at $\partial \Omega$. To circumvent this issue one can reason with the family of functions $\{\mathbf{u}_r \eta\}$, where $\eta \in C^\infty_0(\Omega)$ is a positive smooth cutoff. We refer to \cite{SoZi} for further details. This is the reason why the uniform estimate in Theorem \ref{thm:unif Lip 2} is only of local type (true in any compact $K \subset \Omega$).
\end{remark}


\end{document}